\documentclass[a4paper,11pt]{article}
\usepackage{fullpage}

\usepackage[utf8]{inputenc}
\usepackage[english]{babel}

\usepackage{amsfonts, amssymb, amsbsy, amsmath, amsthm}
\usepackage{dsfont}

\usepackage[sort,numbers]{natbib}
\usepackage[hyphens]{url}
\usepackage{xcolor}
\definecolor{darkblue}{rgb}{0,0,0.7}
\usepackage[colorlinks=true, allcolors=darkblue]{hyperref}
\usepackage[hyphenbreaks]{breakurl}
\usepackage[capitalise]{cleveref}
\hypersetup{
    colorlinks=true,
    breaklinks=true,            %3
}

\usepackage{algorithm2e}
\RestyleAlgo{ruled}
\Crefname{algocf}{Algorithm}{Algorithms}
\Crefname{equation}{Equation}{Equations}
\Crefname{figure}{Figure}{Figures}

\usepackage{pgfplots}
\usetikzlibrary{positioning}
\usepackage{caption}
\usepackage{wrapfig}

\usepackage{setspace}

\usepackage{listings}

\newtheorem{theorem}{Theorem}[section]
\newtheorem{lemma}[theorem]{Lemma}
\newtheorem{definition}[theorem]{Definition}

\newtheorem{proposition}[theorem]{Proposition}

\theoremstyle{remark}
\newtheorem*{example}{Example}

\newcommand{\mQ}{\mathcal{Q}}
\newcommand{\mP}{\mathcal{P}}
\newcommand{\mN}{\mathcal{N}}
\newcommand{\mX}{\mathcal{X}}
\newcommand{\mF}{\mathcal{F}}
\newcommand{\mG}{\mathcal{G}}
\newcommand{\mH}{\mathcal{H}}

\newcommand{\mB}{\mathcal{B}}

\newcommand{\RR}{\mathbb{R}}
\newcommand{\NN}{\mathbb{N}}
\newcommand{\ZZ}{\mathbb{Z}}
\newcommand{\DS}{\mathbb{S}}

\newcommand{\bw}{{\bf w}}
\newcommand{\bs}{{\bf s}}
\newcommand{\bx}{{\bf x}}

\DeclareMathOperator*{\argmin}{arg\,min}

\usepackage{authblk}
\usepackage{todonotes}

\title{New Ramsey Multiplicity Bounds and Search Heuristics}
\author[2,3]{Olaf Parczyk}
\author[1,2]{Sebastian Pokutta}
\author[1,2]{Christoph Spiegel}
\author[3]{Tibor Szab{\'o}}

\affil[1]{\small Technische Universit\"at Berlin, Institute of Mathematics}
\affil[2]{\small Zuse Institute Berlin, Department AIS2T, \emph{lastname}@zib.de}
\affil[3]{\small Freie Universit\"at Berlin, Institute of Mathematics, \emph{lastname}@mi.fu-berlin.de}

\begin{document}

\maketitle

\begin{abstract} 
    We study two related problems concerning the number of homogeneous subsets of given size in graphs that go back to questions of Erd\H{o}s. Most notably, we improve the upper bounds on the Ramsey multiplicity of $K_4$ and $K_5$ and settle the minimum number of independent sets of size $4$ in graphs with clique number at most $4$. Motivated by the elusiveness of the symmetric Ramsey multiplicity problem, we also introduce an off-diagonal variant and obtain tight results when counting monochromatic $K_4$ or $K_5$ in only one of the colors and triangles in the other. The extremal constructions for each problem turn out to be blow-ups of a graph of constant size and were found through search heuristics. They are complemented by lower bounds established using flag algebras, resulting in a fully computer-assisted approach. For some of our theorems we can also derive that the extremal construction is stable in a very strong sense. More broadly, these problems lead us to the study of the region of possible pairs of clique and independent set densities that can be realized as the limit of some sequence of graphs.
\end{abstract}

\section{Introduction} \label{sec:introduction}

In Extremal Combinatorics, the application of computer assistance in formulating proofs is exemplified by flag algebras~\cite{Razborov_2007}, which allow one to establish bounds through a double-counting argument by solving Semidefinite Programming (SDP) problems. These bounds are sometimes complemented by explicit constructions derived from human combinatorial insights. However, as for example noted in~\cite{PikhurkoEtAl2019}, there is a growing interest in computer-based approaches that can complement flag algebras and surpass human intuition. In this work we investigate how the objective of finding constructive bounds can be formulated as a discrete optimization problem and explore the suitability of metaheuristics, such as Simulated Annealing~\cite{KirkpatrickGelattVecchi_1983} and Tabu Search~\cite{Glover_1986, Glover_1989, Glover_1990}, as well as more recent Reinforcement Learning methods~\cite{Wagner_2021, BelloEtAl_2016}, for these problems. 
To illustrate the potential of this approach, we make significant progress on well-known problems in Extremal Graph Theory that go back to questions of Erd\H{o}s. Our investigations also lead us to introduce new directions where many questions remain open. 

\subsection{The Ramsey multiplicity problem}

A subset of vertices in a graph is called \emph{homogeneous} if it is either a clique, so all pairs of vertices form an edge, or it is an independent set, i.e., no pairs of vertices form an edge. 
The existence and enumeration of homogeneous subsets is a fundamental, widely-studied topic in combinatorics.  
Ramsey's Theorem states that every $n$-vertex graph contains a homogeneous $t$-subset, provided that $n$ is large enough compared to $t$.
In 1959, Goodman determined the smallest possible number of homogeneous $3$-subsets an $n$-vertex graph can have for every $n$ and also raised the analogous problem for $t$-subsets when $t \geq 4$.  
A couple of years later, Erd\H os~\cite{Erdos_1962} observed that the number of homogeneous $t$-subsets can be as small as $2^{1-{t\choose 2}}\cdot{n\choose t}$, since this is their expected number in the uniform random graph $G(n,1/2)$, and hence there must also \emph{exist} a graph on $n$ vertices with at most that many homogeneous $t$-subsets. Erd\H os in fact conjectured that this should be the asymptotic minimum for every $t\geq 4$.
This likely was motivated by the symmetry of the extremal function combined with the role of $G(n,1/2)$  in providing large graphs \emph{without} homogeneous $t$-subsets, and was further supported by the fact that Goodman's lower bound happens to agree with $2^{-2}\cdot {n\choose 3}$ asymptotically. 
Subsequently, several positive results were proved establishing the conjecture for graphs with weaker and weaker pseudorandom conditions~\cite{Thomason_1987, EvansPulhamSheehan_1981, FranekRodl_1992}, yet the general conjecture proved to be difficult to crack already for $t=4$, a bit of a surprise considering the beauty and relative simplicity of the double-counting arguments for $t=3$. Thomason~\cite{Thomason_1989}  soundly rejected the conjecture in 1989 for every $t\geq 4$, with counterexamples given by sequences of blow-ups of well-designed constant-size graphs. The finite structure ruling these constructions dismissed any heuristic speculation 
connecting the problem to uniform randomness.

The hunt for the true value of the minimum density of homogeneous $t$-subsets is still ongoing for every $t\geq 4$.  
The determination of this minimum is known as the \emph{Ramsey multiplicity problem} for cliques and it has received a fair amount of attention~\cite{Conlon_2012,deza2012conjecture,Erdos_1962,even2015note,Franek_2002,FranekRodl_1993,Giraud_1979,GrzesikEtAl_2020,jagger1996multiplicities,Niess_2012,Sawin_2021,Sperfeld_2011,Thomason_1989,Thomason_1997, Wolf_2010, EvansPulhamSheehan_1981}.
Subsequently the problem was also extended and investigated for arbitrary graphs, hypergraphs, and other discrete structures.

Our first results make progress on the Ramsey multiplicity problem. 
We denote the number of cliques on $t$ vertices in a graph $G$ by $k_t(G)$ and let $k_t(n) = \min \{k_t(G) + k_t(\overline{G}) : |G| = n\}$ be the minimum number of homogeneous $t$-subsets in an $n$-vertex graph.
The limit 
\begin{equation*}
    c_t = \lim_{n \to \infty} k_t(n) / {n \choose t}
    \end{equation*}
exists for every fixed $t$, as the sequence of minimum proportions $k_t(n) / {n \choose t}$ of homogeneous $t$-subsets is non-decreasing in $n$.

Trivially $c_2=1$ and Goodman's result implies $c_3= \frac{1}{4}$. 
For $t=4$ Thomason~\cite{Thomason_1989} showed $c_4 < 0.030304< 0.03125 = 2^{-5}$, hence disproving Erd\H os' conjecture. 
Later Franek and R\"odl~\cite{FranekRodl_1993} gave another more straightforward counterexample, with a somewhat larger homogeneous $t$-subset density.
In 1997, Thomason~\cite{Thomason_1997} improved the upper bound by $1.3\cdot 10^{-5}$ to $c_4 < 0.030291$. 
Almost thirty years later, Even-Zohar and Linial~\cite{even2015note}  pushed 
it further down by another $0.6\cdot 10^{-5}$ to $c_4 < 0.030285$. 
Here we establish the following more substantial improvement.\footnote{We do sense some absurdity in casting an improvement of $14\cdot10^{-5}$  as {\lq}substantial{\rq}, yet, considering the pace of the progress of the upper bound since Thomason's breakthrough and the gap between upper and lower bounds, we cautiously hope not to be completely out of line.}

\begin{theorem}\label{thm:4-ramsey_multiplicity}
    We have
    \begin{align*}
       c_4 \leq 4551721 \cdot 2^{-24} \cdot 3^{-2} < 0.030145
    \end{align*}
\end{theorem}

The best known lower bound stands at $0.0296 < c_4$, and was proved by Grzesik, Lee, Lidick\'y and Volec~\cite{GrzesikEtAl_2020} using flag algebras, so \cref{thm:4-ramsey_multiplicity} reduces the gap between upper and lower bounds by more than 20\%.
Perhaps more relevant than the actual values are the methods used to find them. The best of 
Thomason's initial constructions~\cite{Thomason_1989} were based on the blow-up of a {\lq}core{\rq} graph on $272$ vertices, which were constructed using quadratic forms in an appropriate finite geometric setting. 
The core graphs of Franek and R\"odl~\cite{FranekRodl_1993} were defined on $1024$ vertices, over the subsets of a $10$-element set using intersection sizes.  
The improvement of Thomason~\cite{Thomason_1997} built on insights from the work of Jagger, {\v{S}}{\'t}ov{\'\i}{\v{c}}ek, and Thomason~\cite{jagger1996multiplicities} about the construction in \cite{Thomason_1989}, and employed an extensive brute force computer search over XOR products of small graphs to find its core graph on $288$ vertices. The improved bound by Even-Zohar and Linial~\cite{even2015note} was obtained by identifying a different, iterative blow-up hidden in the construction of Thomason. We will discuss the particularities of the upper bounds and the underlying constructions in more detail in \cref{sec:proof_ramsey_multiplicity}.

The core graphs of all these constructions happen to be Cayley graphs, yet they were not really sought for as such. 
For our improvement, we directly construct Cayley graphs via various computer search heuristics which target to find their generating set. A key advantage of this approach is that  the size of the search space is only roughly the square root of the size of the produced outcome. Hence our heuristics can efficiently explore and produce relatively large core graphs from the great wealth of good ones among \emph{all} Cayley graphs over a given group (and not only consider a small portion of them, as before).
With our approach, we can find core graph constructions on as few as $192$ vertices that improve upon the previous best bound and that can be found with very moderate computational effort. 
The upper bound for $c_4$ in \cref{thm:4-ramsey_multiplicity} is obtained from the sequence of blow-ups of a Cayley graph on $768$ vertices, though we also found a construction on $384$ vertices that likewise beats $0.03015$, hence performing just marginally worse than the construction on $768$ vertices.  

Much of the effort made by several researchers to optimize Thomason's approach for $t = 4$ materialized in improvements for $t\geq 5$. The best known upper bound for $t=5$ is $c_5 < 0.001720$, due to Thomason~\cite{Thomason_1997}, improving on his bound in \cite{Thomason_1989}. 
 For $t \in \{6, 7, 8\}$ the best known bounds are given by Deza, Franek, and Liu \cite{deza2012conjecture} using the intersection graph idea of \cite{FranekRodl_1993}.

Our method can also be employed to improve the best known upper bound in the $K_5$-multiplicity problem.  We also complement this with a lower bound using flag algebras.

\begin{theorem}\label{thm:5-ramsey_multiplicity}
    We have
    \begin{align*}
       0.001524 < c_5 \leq 2320651 \cdot 2^{-24} \cdot 3^{-4} < 0.001708.
    \end{align*}
\end{theorem}

Our improved upper bound is based on the sequence of blow-ups of a Cayley graph on $192$ vertices, just like Thomason's. We will go more into about our bounds and how they were found in \cref{sec:upper_bounds}.

\subsection{Independent set density for graphs with bounded clique number}

A related question, likewise raised by Erd\H os~\cite{Erdos_1962}, concerns the minimum number $k'_{s,t}(n) = \min \{ k_s(\overline{G}) : |G| = n, k_t(G) = 0 \}$ of independent sets of size $s$ in a graph of order $n$ with clique number bounded by $t-1$. The associated limit, whose existence again follows by monotonicity~\cite{Nikiforov_2001}, is defined as
\begin{equation*}
    g_{s,t} = \lim_{n \to \infty} k'_{s,t}(n) / {n \choose s}.
\end{equation*}
Note that obviously $g_{t,t} \geq c_t$ for any $t \geq 2$. We trivially have $g_{s,2} = 1$ and the fact that $g_{2,t} = 1/(t-1)$ is established by Tur{\'a}n's theorem~\cite{Turan_1941}. Erd\H{o}s~\cite{Erdos_1962} asked if the upper bound given by the Tur{\'a}n graphs $T_{t-1}(n)$ as $n \to \infty$ is also tight in general, that is if $g_{s,t} = (t-1)^{1-s}$.
This holds for $s = t = 3$, as an easy consequence of the previously mentioned result of Goodman~\cite{Goodman_1959} and the value of $k'_{3,3}(n)$ was settled precisely by Lorden~\cite{lorden1962blue}.
Nikiforov~\cite{Nikiforov_2001} however showed that the Tur\'an graph upper bound can be sharp only for a finite number of pairs $s, t \geq 3$. Das et al.~\cite{DasEtAl_2013} and Pikhurko and Vaughan~\cite{PikhurkoVaughan_2013} established tight values when one of the parameters is three and the other at most seven. In particular this confirmed Erd\H{o}s' intuition for $g_{3, t}$ when $4 \leq t \leq 7$ and disproved it for $g_{s, 3}$ when $4 \leq s \leq 7$. They also determine the unique extremal constructions in these cases. Moreover, Pikhurko and Vaughan~\cite{PikhurkoVaughan_2013} found a construction based on non-balanced blow-ups of a $(3,4)$-Ramsey graph of order $8$ bounding $g_{4,4}$ away from the value given by $T_3(n)$ that is conjectured to be tight.
Here we present the first tight value for $g_{s,t}$ when $s,t \ge 4$.
Furthermore, we show that any graph that comes close to the optimum of $g_{4,5}$ must be close to a balanced blow-up of $C_{R(3,5)}$, the unique $(3,5)$-Ramsey graph of order $13$, i.e., the Cayley graph on $\ZZ_{13}$ whose edge relations are given by the cubic-non-residues.
\begin{theorem}\label{thm:axis_ramsey_multiplicity}
    We have $g_{4, 5} = 29 \cdot 13^{-3}$ and the problem is perfectly $C_{R(3,5)}$-stable.
\end{theorem}
The upper bound is given by the sequence of balanced blow-ups of $C_{R(3,5)}$, the lower bound matching that construction was established using the flag algebra approach. The notion of perfect stability was introduced by Pikhurko, Slia\v{c}an, and Tyros~\cite{PikhurkoEtAl2019} and strengthens the standard notion of stability. We will formally state it in \cref{def:stability} in \cref{sec:stability_theory}. For the details of the proofs see \cref{sec:lower_bounds}.

The fact that Ramsey graphs are a good source of constructions for this problem was previously already noted by Nikiforov~\cite{Nikiforov_2001} and Das et al.~\cite{DasEtAl_2013}. We found several more such graphs whose sequence of (sometimes non-balanced) blow-ups give good upper bounds for additional values of $g_{s,t}$, but were unable to establish matching lower bounds. The most reasonable open conjecture out of all studied values seems to be that $g_{5,5} = 61 \cdot 13^{-4}$, where the upper bound also comes from $C_{R(3,5)}$. We list all other bounds in \cref{sec:constructions}. We also remark that more general bounds for $g_{s,t}$ were given by Nikiforov~\cite{Nikiforov_2001} and Sawin~\cite{Sawin_2021}, who studied the close connection to (multicolor) Ramsey numbers.

\subsection{Off-diagonal Ramsey multiplicity}

There is a stark contrast between the difficulty of counting homogeneous triples and counting homogeneous $4$-subsets.
This is demonstrated by the relatively straightforward proof of $c_3=1/4$ and the slow progress towards determining the value of $c_4$.
To get a grip on the latter, we propose to investigate the question when we count $K_4$ only in the graph and triangles in the complement.
This problem proves to be more manageable, as we can not only derive an exact solution but also demonstrate the stability of a unique construction on $27$ vertices.

Formally, for arbitrary $s,t\geq 3$, we propose to study the following off-diagonal version of the Ramsey multiplicity parameter:
\begin{equation} \label{eq:cst}
    c_{s,t} = \lim_{n \to \infty} \min \left\{ \frac{k_s(\overline{G})}{{n \choose s}} +  \frac{k_t(G)}{{n \choose t}} : |G| = n \right\}.
\end{equation}
Observe that clearly $c_{t,t}=c_t$ as well as $c_{s,t} \leq \min \{ g_{s,t}, g_{t,s} \}$.
From a result of Reiher~\cite{Reiher_2016} it follows that $c_{2,t} = g_{2,t}$ for every $t \geq 3$.
Here we establish the first exact results when $t> s\geq 3$.
\begin{theorem}\label{thm:offdiagonal_ramsey_multiplicity}
    We have $c_{3,4} = 689 \cdot 3^{-8}$ and the problem is perfectly $C_{S}$-stable, where $C_S$ is the Schl\"afli graph. We also have $c_{3,5} = 24011 \cdot 3^{-12}$ as well as $0.007688 < c_{4,5} \leq 0.007932$.
\end{theorem}
Note that $c_{s,t} < g_{s,t}$ for each of these values of $s$ and $t$. The upper bound for $c_{3,4}$ is uniquely given by the sequence of balanced blow-ups of the Schl\"afli graph. The upper bound for $c_{3,5}$ is more easily described as an upper bound of $c_{5,3}$, in which case it is given by the sequence of balanced blow-ups of the complement of the Schl\"afli graph. The upper bound for $c_{4,5}$ is given by the sequence of blow-ups of a vertex-transitive graph on $128$ vertices. Lower bounds are again established using the flag algebra approach. The proof of stability largely follows the template laid out by Pikhurko et al.~\cite{PikhurkoEtAl2019}, but also requires additional ideas.

A central question in the symmetric Ramsey multiplicity problem is whether a tight upper bound can be achieved through the sequence of (possibly weighted or iterated) blow-ups of a finite-sized graph. For $c_3$, this is true, with Goodman's bound reached by, among others, the blow-ups of $K_2$. However, for $c_4$, the question remains open. Theorem~\ref{thm:offdiagonal_ramsey_multiplicity}, which states the extremality and stability of blow-ups of a single finite graph (on 27 vertices) for the $c_{3,4}$ problem, may suggest the existence of such a finite graph for $c_4$. Our results for $c_4$ however suggest that any such finite construction might be of considerable size.

There is no specific reason for our choice to weight the contributions from both terms equally in $c_{s,t}$. Other choices, such as the weighting given by linearly connecting the points indicated by $g_{s,t}$ and $g_{t,s}$, would also be reasonable.
In fact, the previously introduced problems are part of a broader question aiming to understand which pairs of clique and independent set densities can be realized as the limit of some sequence of graphs. We introduce this problem and our results in the next section.

\medskip

\noindent \textbf{Outline.} We first explore the broader question of realizable pairs of clique and independent set densities and their relation to earlier introduced parameters in \cref{sec:full_region}. Next, we describe the constructions for upper bounds and their discovery methods in \cref{sec:upper_bounds}. We then outline Razborov's flag algebra approach and its application for stability results in \cref{sec:lower_bounds}. Finally, we discuss related problems and recent learning-based optimization heuristics in \cref{sec:discussion}.

\section{The full tradeoff between cliques and independent sets}\label{sec:full_region}

The study of $c_{s,t}$ and $g_{s,t}$ are part of a broader question in which one would like to understand the full tradeoff between the number of cliques of size $t$ and independent sets of size $s$ in a graph. The goal is to characterise the region $\Omega_{s,t} \subseteq [0,1]^2$ of pairs of clique and independent set densities that can occur in the limit of a sequence of graphs.
We say that a tuple $(x,y) \in [0,1]^2$ is \emph{realised by a sequence} of graphs $(G_n)_{n \in \mathbb{N}}$ with $\lim_{n \to \infty} v(G_n)=\infty$, where w.l.o.g. and to simplify notation we assume $v(G_n)=n$, if
\begin{align*}
    \lim_{n \to \infty} k_s(\overline{G_n})/\binom{n}{s} = x \quad \text{and} \quad \lim_{n \to \infty} k_t(G_n)/\binom{n}{t} = y.
\end{align*}
We formally define $\Omega_{s,t} \subseteq [0,1]^2$ to be the set of all tuples that are realised by some sequence of graphs.

For $s=2$ and with $K_t$ replaced by any quantum graph, that is an arbitrary linear combination of graphs, this set was already systematically studied by Liu, Mubayi, and Reiher~\cite{LiuMubayiReiher_2021}. Similar to them, let us define
\begin{equation*}
     c_{s,t}(x) = \inf\{ y \colon (x,y) \in \Omega_{s,t} \} \quad \text{and} \quad  C_{s,t}(x)=\sup\{ y \colon (x,y) \in \Omega_{s,t}\}.
\end{equation*}
We can show that $\Omega_{s,t}$ behaves nicely for any $s,t \geq 2$, that is it is completely characterised by its lower and upper bounding curves $c_{s,t}(x)$ and $C_{s,t}(x)$.
\begin{proposition}
    \label{thm:region}
    $\Omega_{s,t}$ is compact and defines a simply connected region for any $s,t \ge 2$.
\end{proposition}
\begin{proof}
     We start by establishing compactness of $\Omega_{s,t}$ as in Proposition~1.3 from~\cite{LiuMubayi2021}.
    Let $(x_m,y_m)_{m \in \mathbb{N}}$ be a sequence in $\Omega_{s,t}$ such that $\lim_{m \to \infty}(x_m,y_m) = (x,y)$ and let us show that $(x,y) \in \Omega_{s,t}$.
    For each $m \in \mathbb{N}$ there exists a sequence of graphs $(G_{n,m})_{n \in \mathbb{N}}$ that realises $(x_m,y_m)$.
    With $x_{n,m} = k_s(\overline{G_{n,m}})/\binom{n}{s}$ and $y_{n,m} = k_t(G_{n,m})/\binom{n}{t}$ we have a sequence $(x_{n,m},y_{n,m})$ with $\lim_{n \to \infty}(x_{n,m},y_{n,m}) = (x_m,y_m)$.
    Therefore, by Lemma~2.2 from~\cite{LiuMubayi2021}, there exists $(n_k)_{k \in \mathbb{N}}$ such that $\lim_{k \to \infty}(x_{n_k,k},y_{n_k,k}) = (x,y)$ and, thus, the sequence $(G_{n_k,k})_{k \in \mathbb{N}}$ realises $(x,y)$.
    
    Now let us show that the region $\Omega_{s,t}$ is simply connected by showing that for any $(x,y_1)$, $(x,y_2) \in \Omega_{s,t}$ we must also have $(x,y) \in \Omega_{s,t}$ for any $y_1 \leq y \leq y_2$.
    We generalise the argument in Proposition~2.2 from~\cite{LiuMubayiReiher_2021} and consider sequences $(G_{n})_{n \in \mathbb{N}}$ and $(G'_{n})_{n \in \mathbb{N}}$ that realise $(x,y_1)$ and $(x,y_2)$ respectively.
    For each $n \in \mathbb{N}$ we iteratively construct a sequence of $n$-vertex graphs $G_{n,1},\dots,G_{n,k(n)}$ with $G_{n,1}=G_n$ and $G_{n,k(n)}=G'_n$ as follows:
    for $i \ge 1$ if $G_{n,i}=G'_n$, then we set $k(n)=i$ and stop.
    If $k_s(\overline{G_{n,i}})/\binom{n}{s} \ge x$ and $E(G'_n) \setminus E(G_{n,i})\not=\emptyset$ we obtain $G_{n,i+1}$ from $G_{n,i}$ by adding any edge from $E(G'_n) \setminus E(G_{n,i})$.
    Similarly, if $k_s(\overline{G_{n,i}})/\binom{n}{s} < x$ and $E(G_{n,i}) \setminus E(G_n')\not=\emptyset$ we obtain $G_{n,i+1}$ from $G_{n,i}$ by removing any edge from $E(G_{n,i}) \setminus E(G_n')$.
    Otherwise, either $E(G_{n,i}) \setminus E(G_n')=\emptyset$ or $E(G'_n) \setminus E(G_{n,i})=\emptyset$ and we obtain $G_{n,i+1}$ from $G_{n,i}$ by adding an edge of $E(G'_n) \setminus E(G_{n,i})$ in the former case and removing one of $E(G_{n,i}) \setminus E(G_n')$ in the latter.
    We note that adding or removing a single edge adds or removes a fraction of cliques of a given size that is $o(1)$, so that $\lim_{n \to \infty} k_s(\overline{G_{n,i(n)}})/\binom{n}{s} =x$ and $\lim_{n \to \infty} \left( k_t({G_{n,i(n)}}) - k_t({G_{n,i(n)+1}}) \right) /\binom{n}{t} =0$ for any sequence $i(n)$ satisfying $1 \leq i(n) \leq k(n)-1$.
    Therefore, as $\lim_{n \to \infty} k_t({G_{n,1}})/\binom{n}{t} = y_1$ and $\lim_{n \to \infty} k_t({G_{n,k(n)}})/\binom{n}{t} = y_2$, for any $y$ with $y_1 \leq y \leq y_2$ there exists $k'(n)$ such that $\lim_{n \to \infty} k_t({G_{n,k'(n)}})/\binom{n}{t} = y$.
\end{proof}

It follows that it suffices to study $c_{s,t}(x)$ and $C_{s,t}(x)$ in order to fully understand $\Omega_{s,t}$ and we can in fact replace the infimum and supremum in their definition by a minimum and maximum. Let us establish some properties of these curves.

\begin{proposition} \label{thm:curves}
     The curves $c_{s,t}(x)$ and $C_{s,t}(x)$ are decreasing, continuous, and almost everywhere differentiable for any $s,t \geq 2$.
\end{proposition}

\begin{proof}
    We start by arguing that both $c_{s,t}(x)$ and $C_{s,t}(x)$ are decreasing, similarly to Lemma~2.3 in~\cite{LiuMubayiReiher_2021}.
    Let $(G_n)_{n \in \NN}$ be a sequence of graphs of order $n$ realizing a point $(x, y) \in [0,1]^2$ with $x < 1$.
    Consider the sequence $(G_n')_{n \in \NN}$ obtained by fully connecting a clique of size $\lceil \beta \, n \rceil$ to $G_n$ for some fixed $\beta > 0$ and any $n \in \NN$.
    Let us determine the tuple $(x',y')$ realized by that sequence.
    As no new independent sets of size $s$ are created in $G_n'$ we have $k_s(\overline{G_n'}) = k_s(\overline{G_n})$.
    On the other hand, for each $0 \le \ell \le t-1$, we get at least $k_\ell(G) \cdot \binom{\lceil \beta n \rceil}{t-\ell}$ new $K_t$ in $G_n'$ and hence 
    \[
        k_t({G'_n}) \ge \sum_{\ell=0}^t k_\ell({G_n}) \binom{\lceil \beta n \rceil}{t-\ell} \, .
    \]
    As every copy of $K_t$ in $G_n'$ was already contained in $G_n$ or contains at least one of the new vertices, we have $k_t({G'_n}) \le k_t({G_n})+ \lceil \beta n \rceil \binom{n+\lceil \beta n \rceil}{t-1}$.
    Combining this, dividing by $\binom{n+\lceil \beta n \rceil}{s}$ and $\binom{n+\lceil \beta n \rceil}{t}$ respectively, and taking limits, we get
    \begin{equation}
    \label{eq:xyprimexy}
     y' = \frac{y}{(1+\beta)^s} \quad \text{and} \quad
        \frac{x + t\beta \, (1 + \beta)^{t-1}}{(1 + \beta)^t} \ge x' \ge \frac{\beta^t + t \beta^{t-1}+\sum_{\ell=2}^t \binom{t}{\ell} x \beta^{t-\ell}}{(1 + \beta)^t} \, ,
    \end{equation} 
    where for the last inequality we use that $\liminf_{n \to \infty}k_\ell(G_n)\binom{\lceil \beta n \rceil}{t-\ell}/\binom{n}{t} \ge x \binom{t}{\ell} \beta^{t-\ell}$.
    Using these bounds we can easily find a $\beta = \beta(\varepsilon) > 0$ for any $\varepsilon > 0$ such that $x < x' \leq x + \varepsilon$, where $x<x'$ holds as $\beta^t + t \beta^{t-1}+\sum_{\ell=2}^t \binom{t}{\ell} x \beta^{t-\ell} > x (1+\beta)^t$ by the binomial Theorem.
    Since $y' < y$ for $\beta > 0$, the fact that $c_{s,t}(x)$ is decreasing follows.
    A similar argument likewise establishes that $C_{s,t}(x)$ is decreasing.
    
    Given the monotonicity of $c_{s,t}(x)$ and $C_{s,t}(x)$ as well as their bounded domain of $[0,1]$, the fact that they are almost everywhere differentiable immediately follows.
    Regarding continuity, we note that both left- and right-hand limits exist due to monotonicity.
    Since $c_{s,t}$ is the decreasing lower bounding curve of a compact domain, it must also be right-continuous and left-continuity of $C_{s,t}(x)$ likewise follows.
    To establish left-continuity for $c_{s,t}(x)$ (and right-continuity of $C_{s,t}(x)$) we let $x_0 \in (0,1)$ and $y_0=\lim_{x \nearrow x_0} c_{s,t}(x)$.
    By monotonicity $c_{s,t}(x_0) \le y_0$, so let us assume that $c_{s,t}(x_0)=y' < y_0$.
    Let $\beta>0$ be small enough such that $y_0/(1+\beta)^s > y'$.
    Then, in view of \cref{eq:xyprimexy}, choose $\alpha>0$ small enough such that
    \[\frac{\beta^t+t \beta^{t-1} + \sum_{\ell=2}^t \binom{t}{\ell} (x_0-\alpha) \beta^{t-\ell}}{(1+\beta)^t} > x_0 \, ,\]
    which is again possible by the binomial Theorem.
    As $c_{s,t}(x_0-\alpha) \ge y_0$ we get with monotonicity and \cref{eq:xyprimexy} that $c_{s,t}(x_0)\ge y_0/(1+\beta)^s > y'$.
    This contradicts our assumption and establishes left-continuity of $c_{s,t}(x)$.
\end{proof}

$C_{s,t}(x)$ is the easier of the two to establish and when $\min \{s,t\} = 2$ is precisely given by the Kruskal-Katona theorem, which states that if $k_s(G) / {n \choose s} = \alpha$ then $k_t(G) / {n \choose t} \leq \alpha^{t / s}$~\cite{Kruskal1963, Katona1968} so that for $s = 2$ if $k_2(\overline{G}) / {n \choose 2} = x$ and hence $k_2(G) / {n \choose 2} = 1-x$, then $k_t(G) / {n \choose t} \leq (1-x)^{t / 2}$. It was more generally determined for arbitrary $s,t \geq 2$ by Huang et al.~\cite{Huang_2016}, who showed that the maximum is always achieved either by a clique with additional isolated vertices or by the complement of this graph.
More precisely,
\begin{equation}
    \label{eq:Cst}
    C_{s,t}(x) = \max\{(1-x^{1/s})^t+t \, x^{1/s} \, (1-x^{1/s})^{t-1} , \, (1-z)^t \},
\end{equation}
where $z$ is the unique root of $z^s + sz^{s-1}(1-z ) = x$ in $[0,1]$.
Note that this is differentiable except for at the point where the curves given by the two constructions meet.

\begin{figure}
\captionsetup{width=0.4\textwidth, font=scriptsize}

    \begin{minipage}[b]{.49\textwidth}
    \centering
    \begin{tikzpicture}[scale=50]
        \draw[scale=0.12, domain=0:1, smooth, variable=\x, black] plot ({\x}, {(1-\x)^(4/2)});
        \draw [<->,thick] (0,0.13) node (yaxis) [above] {$y$} |- (0.13,0) node (xaxis) [right] {$x$};
        \coordinate (a1) at (.06,0);
        \coordinate (a2) at (.06/2,0.015);
        \coordinate (a3) at (.06/3,0.023);
        \coordinate (a4) at (.06/4,0.031);
        \coordinate (a5) at (.06/5,0.039);
        \coordinate (a6) at (.06/6,0.047);
        \coordinate (a7) at (.06/7,0.055);
        \coordinate (a0) at (0,0.12);
        \coordinate (a8) at (0.12,0);
        
        \draw[thick, black] (a1) to[out=150,in=-20] (a2);
        \draw[thick, black] (a2) to[out=130,in=-30] (a3);
        \draw[thick, black] (a3) to[out=110,in=-40] (a4);
        \draw[thick, black] (a4) to[out=100,in=-50] (a5);
        \draw[thick, black] (a5) to[out=95,in=-60] (a6);
        \draw[thick, black] (a6) to[out=90,in=-70] (a7);
        
        \draw[dashed] (a7)--(a0);
        
       \foreach \i in {0,...,8}
       {\fill[red] (a\i) circle (0.03pt);}
        
        \node[right = 0.1 of a0] {\tiny $(0,1)$};
        \node[above = 0.1 of a8] {\tiny $(1,0)$};
        \node[right = 0.01 of a4] {\tiny $\left(r^{-1},\tfrac{r^{\underline{t-1}}}{r^{t-1}}\right)$
        };
        \node[above = 0.1 of a1] {\tiny $\left(\tfrac{1}{t-1},0\right)$};
    \end{tikzpicture}
    \captionof{figure}{The whole region $\Omega_{2,t}$.
    The lower bound $c_{2,t}(x)$ as determined by Razborov~\cite{Razborov_2008}, Nikiforov~\cite{Nikiforov_2011}, and Reiher~\cite{Reiher_2016} and the upper bound $C_{2,t}(x)=(1-x)^{t/2}$ follows from \cref{eq:Cst}.
    The red dots are given by a balanced complete $r$-partite graph for $r=t-1,\dots$ and the connections come from an interpolation between them.
    }
    \label{fig:c_23}
    \end{minipage}
    \begin{minipage}[b]{.49\textwidth}
	    \centering
	     \begin{tikzpicture}[scale=50]
	     
	     \draw[scale=0.12, domain=0:.65, smooth, variable=\x, black] plot ({\x*\x*\x}, {(1-\x)*(1-\x)*(1-\x)+3*(1-\x)*(1-\x)*\x});
        \draw[scale=0.12, domain=0:.65, smooth, variable=\x, black] plot ({(1-\x)*(1-\x)*(1-\x)+3*(1-\x)*(1-\x)*\x},{\x*\x*\x});
	     
        \draw [<->,thick] (0,0.13) node (yaxis) [above] {$y$
        } |- (0.13,0) node (xaxis) [right] {$x$
        };
        \coordinate (a1) at (0,0.03);
        \coordinate (a2) at (0,0.12);
        \coordinate (a3) at (0.0335,0.0335);
        \coordinate (a4) at (0.12,0);
        \coordinate (a6) at (0.03,0); (a1)--(a2)--(a3)--(a4)--(a5)--(a5b)--(a6);
        \draw[thick] (0,0.03)--(0.03,0);
        
        \fill[red] (a1) circle (0.03pt);
        \fill[red] (a2) circle (0.03pt);
        \fill[red] (a3) circle (0.03pt);
        \fill[red] (a4) circle (0.03pt);
        \fill[red] (a6) circle (0.03pt);
        
        \node[right = 0.1 of a1] {\tiny $\left(0,\tfrac{1}{4}\right)$};
        \node[right = 0.1 of a2] {\tiny $\left(0,1\right)$};
        \node[right = 0.1 of a3] {\tiny $\sim\left(0.278,0.278\right)$};
        \node[above = 2pt of a4] {\tiny $\left(1,0\right)$};
        \node[above right = 2pt and 0pt of a6] {\tiny $\left(\tfrac14,0\right)$};
    \end{tikzpicture}
    \captionof{figure}{The whole region $\Omega_{3,3}$ as described by Huang et~al.~\cite{huang_2014}.
    The lower bound is $c_{3,3}(x) = 1/4$ for $x \in [0,1/4]$ as in \cref{obs:c_33}.
    The upper bound $C_{3,3}(x)$ is given by the two functions from \cref{eq:Cst} intersecting at the red dot, where it is not differentiable.
    \\
        }
    \label{fig:c_33}
    \end{minipage}
\end{figure}   

On the other hand, much less is known about $c_{s,t}(x)$.
Clearly $c_{s,t}(0)=g_{s,t}$ and $c_{s,t}(x)=0$ if and only if $x \ge g_{t,s}$, that is $(0, g_{s,t})$ and $(g_{t,s}, 0)$ are the points where the curve $c_{s,t}(x)$ intersects the axes when $0 \leq x \leq g_{t,s}$.
Moreover, $c_{s,t} = \min_{x} c_{s,t}(x) + x$ and therefore $c_{s,t}(x) \ge c_{s,t}-x$. Given \cref{thm:region} and \cref{thm:curves}, we can also equivalently define it without needing to introduce $\Omega_{s,t}$ as
\begin{equation*}
    c_{s,t}(x) = \lim_{n \to \infty} \min \left\{ \frac{k_s(\overline{G})}{\binom{n}{s}} : |G|=n, \frac{k_t(G)}{\binom{n}{t}}\leq x \right\}.
\end{equation*}
For $s = 2$ one is interested in the minimum possible density of cliques of size $t$ in a graph of given edge density and in this case $c_{2,t}(x)$ was completely determined; Razborov~\cite{Razborov_2008} gave an answer for $t=3$, Nikiforov~\cite{Nikiforov_2011} for $t=4$, and Reiher~\cite{Reiher_2016} for arbitrary $t$.
See \cref{fig:c_23} for an illustration of $c_{2,t}(x)$ and $C_{2,t}(x)$.
We note that Liu, Pikhurko, and Staden~\cite{Hong_2017} establish stability and exactness of $c_{2,3}(x)$ for $x \in [1/2,1]$ improving on previous results of Pikhurko and Razborov~\cite{Pikhurko_2017} and Lovász and Simonovits~\cite{Lovasz_1983}, where the latter also covers $c_{2,t}(x)$ at and slightly below the Turán constructions.
A simple deterministic construction also shows that the lower bound by Goodman for $s=t=3$ is tight in this more general case, establishing that $c_{3,3}(x)$ is in fact linear and can be obtained by interpolating between extremal constructions similar as for $c_{2,t}(x)$.

\begin{lemma}
    \label{obs:c_33}
    We have $c_{3,3}(x) = 1/4 -x$ for all $x \in [0,1/4]$.
\end{lemma}

\begin{proof}
    A possible construction of a sequence of $n$-vertex graphs goes as follows.
    For some $\eta \in [0,1/2]$ we take two cliques of size $\lceil \eta n \rceil$ and two independent sets of size $\lfloor (1/2-\eta)n \rfloor$.
    Each clique is fully connected to a different one of the independent sets and also the independent sets are fully connected.
    The density of triangles in this graph approaches $2 \eta^3 + 6 \eta^2(1/2-\eta)$ and the density of independent sets of size three approaches $2 (1/2-\eta)^3 + 6 (1/2-\eta)^2\eta$. Each of these quantities covers $[0,1/4]$ for $\eta \in [0,1/2]$ and together they sum up to $1/4$.
\end{proof}

Together with $C_{3,3}(x)$ as given by~\cite{Huang_2016}, this determines $\Omega_{3,3}$ as illustrated in \cref{fig:c_33}. This region was in fact already earlier established by Huang et~al.~\cite{huang_2014}, who relied on a probabilistic construction for every point in the region, rather than interpolating between the two bounding curves as in the proof of \cref{thm:region}.
Beyond that very little is known about the shape of $c_{s,t}(x)$, providing some additional motivation for studying the parameters $c_{s,t}$ and $g_{s,t}$ stated in the introduction.

Recall that $c_{s,t}(x) \ge c_{s,t}-x$, so \cref{thm:4-ramsey_multiplicity}, \cref{thm:5-ramsey_multiplicity}, and \cref{thm:offdiagonal_ramsey_multiplicity} imply linear lower bounds of $c_{s,t}(x)$ for some specific values of $s$ and $t$.
Let us take a closer look at the smallest open case, that is $s=3$ and $t=4$. 
Calculating the $\overline{K_3}$ and $K_4$ density of the sequence of blow-ups\footnote{A formal definition of a blow-up will be given in the next section in \cref{def:blowup}.} of the Schl\"afli graph, we get that $c_{3,4}(41 \cdot 3^{-6}) = 320 \cdot 3^{-8}$, establishing one precise value of the curve besides the ones on the axes.
Additionally, we can show that $c_{3,4}(x)$ is not differentiable at this point by establishing a second tight lower bound using a differently weighted version of $c_{3,4}$, see \cref{sec:flag_algebra_practical} for more details.

Noting that by Das et al.~\cite{DasEtAl_2013} and Pikhurko and Vaughan~\cite{PikhurkoVaughan_2013} the value of $g_{3,4} = c_{3,4}(0) = c_{4,3}(1/9)$ is determined by the sequence of blow-ups of $K_3$ and the value of $g_{4,3} = c_{3,4}(3/25) = c_{4,3}(0)$ by that of a $C_5$ with loops at all vertices, it also seems reasonable to ask if the Schl\"afli graph could mark a first {\lq}extremal point{\rq} of the curve $c_{3,4} (x)$ that does not lie on either axis, the same way that $K_{t}$ does for $c_{2,t}(x)$, or if alternatively the blow-up sequence of another vertex-transitive graph on fewer than $27$ vertices marks such a point.\footnote{We are intentionally vague here with our notion of extremal points. We likely expect a, possibly infinite but countable, set of points $x$ in $[0,1]$ where $c_{3,4}(x)$ is not differentiable and that behavior corresponds to a change in the underlying graph construction.}

\begin{wrapfigure}{l}{.5\textwidth}
\captionsetup{width=.5\textwidth, font=scriptsize}
     \begin{tikzpicture}[scale=25]
     
     \draw[scale=.51, domain=.385:.79, smooth, variable=\x, black] plot ({\x*\x*\x}, {(1-\x)*(1-\x)*(1-\x)*(1-\x)+4*(1-\x)*(1-\x)*(1-\x)*\x});

    \draw [<->,thick] (0,0.26) node (yaxis) [above] {$y$} |- (0.26,0) node (xaxis) [right] {$x$};
    \coordinate (a1) at (0,3/25);
    \coordinate (a2) at (3/200*.8,0.099*.85);
    \coordinate (a3) at (1/36*.8,577/6912*.75);
    \coordinate (a4) at (41/729*.75,320/6561*.75);
    \coordinate (a5) at (0.075*.8,0.03*.8);
    \coordinate (a5b)at (0.08*.8,0.027*.8);
    \coordinate (a6) at (1/9,0);
    \draw [dashed] (a1)--(a2)--(a3)--(a4)--(a5)--(a5b)--(a6);
    \draw[thick, blue] (0,689/6561*.75)--(689/6561*.75,0);
    \draw[thick, dotted, blue] (0,0.098*.75)--(0.112*.75,0);
    
    \fill[red] (a1) circle (0.06pt);
    \fill[gray] (a2) circle (0.06pt);
    \fill[gray] (a3) circle (0.06pt);
    \fill[red] (a4) circle (0.06pt);
    \fill[gray] (a5) circle (0.06pt);
    \fill[gray] (a5b) circle (0.06pt);
    \fill[red] (a6) circle (0.06pt);
    
    \node[above right = 0.1 of a1,fill=white,inner sep=0pt] {\tiny $\left(0,\tfrac{3}{25}\right)$ - looped $C_5$};
    \node[above right = 0.05 of a2,fill=white,inner sep=0pt] {\tiny $\left(\tfrac{3}{200},\tfrac{6347}{64000}\right)$ - $40$ vertices};
    \node[above right = 0.05 of a3,fill=white,inner sep=0pt] {\tiny $\left(\tfrac{1}{36},\tfrac{577}{6912}\right)$ - $24$ vertices};
    \node[above right = 5pt and -5pt of a4,fill=white,inner sep=0pt] {\tiny $\left(\tfrac{41}{729},\tfrac{320}{6561}\right)$ - Schl\"afli graph};
    \node[above right = 2pt and 0pt of a5,fill=white,inner sep=0pt] {\tiny $\left(\tfrac{563}{8192},\tfrac{2469}{65536}\right)$ - $128$ vertices};
    \node[above right = -7pt and  10pt of a5b,fill=white,inner sep=0pt] {\tiny $\left(\tfrac{437}{6272},\tfrac{33}{896}\right)$ - $112$ vertices};
    \node[above right = -3pt and 8pt of a6] {\tiny $\left(\tfrac19,0\right)$ - $K_3$};
\end{tikzpicture}
\caption{The known bounds on the region $\Omega_{3,4}$. The blue line is the linear lower bound $y=c_{3,4}-x$, the blue dotted line is an additional linear lower bound for $c_{3,4}(x)$, the red dots represent optimal constructions, and the grey dots represent additional constructions. The upper bound $C_{3,4}(x) = (1-x^{1/3})^4+4x(1-x^{1/3})^3$ follows from \cref{eq:Cst}.}
\label{fig:c_34}
\end{wrapfigure}

Surprisingly, some experimentation reveals that neither option seems to hold true:  there is no sequence of blow-ups of a vertex-transitive graph on up to 47 vertices that determines a point in convex position with the points given by the Schl\"afli graph and $K_3$, where one might expect points of discontinuity to be more easily described. There are however various blow-up sequences of vertex transitive graphs on at least $112$ vertices that are in convex position with those points. On the other hand, between the points given by the Schl\"afli graph and the looped  $C_5$, where we would expect constructions to become increasingly complex, we first find a vertex-transitive graph on only $24$ vertices determining a point in convex position with the two other points. Lest one assume that this might indicate a pattern, there also exists a vertex-transitive graph on $40$ vertices that determines a point in convex position with this point and the one given by the looped $C_5$. Our results are illustrated in \cref{fig:c_34}.
We note that stability also holds for $g_{3,4}$ and $g_{4,3}$~\cite{DasEtAl_2013,PikhurkoVaughan_2013} and \cref{thm:offdiagonal_ramsey_multiplicity} establishes the same for the Schläfli graph.

\section{Constructive upper bounds} \label{sec:upper_bounds}

All upper bounds on $c_t$, $c_{s,t}$, $g_{s,t}$, and $c_{s,t}(x)$ described in the introduction and previous section rely on explicit constructions. When only the asymptotic value is of interest, this usually implies considering the blow-up of a fixed finite graph to obtain a sequence of graphs with increasing order that have some desired property.

\begin{definition}\label{def:blowup}
    For a graph $C$ on vertex set $\{1, \ldots, n\}$ and integers $m_1, \ldots, m_n \in \NN_0$, the \emph{blow-up} $C[m_1, \ldots, m_n]$ is the graph with vertex set $\bigcup_{i=1}^{n} \{(i,v) : 1 \leq v \leq m_i \}$ in which two distinct vertices $(i_1, v_1)$ and $(i_2, v_2)$ are connected if and only if $i_1$ and $i_2$ are connected in $C$. When $|m_i - m_j| \leq 1$ for $1 \leq i,j \leq n$ we say that the blow-up is \emph{balanced}. When $m_i = m$ for $1 \leq i \leq n$, then we call $C[m] = C[m_1, \ldots, m_n]$ the \emph{$m$-fold blow-up} of $C$ and refer to $(C[m])_{m \in \NN}$ as the \emph{blow-up sequence of $C$}.
\end{definition}

This means that a vertex of $C$ is replaced by an independent set of size $m$ in $C[m]$ if that vertex does not have a loop, or by a clique of size $m$ otherwise. Note that our definition ensures that the blow-up is always a simple graph even when $C$ contains loops. It follows that the complement of the blow-up of a graph $C$ is equal to the blow-up of the complement of $C$ with loops added at every vertex. We denote this type of complement as the \emph{looped complement} of a graph.\footnote{It is actually significantly easier to think about these problems in terms of monochromatic homomorphic copies of fully-looped cliques in a two-coloring of a fully-looped (and possibly vertex-weighted) complete graph, but to stay consistent with previous literature we are counting subgraphs of a simple graph and its complement.}

The blow-up of a graph has certain properties that are favorable to this type of problem. In particular, if the graph has no loops then its clique number is equal to that of $C$, that is $\omega(C[m]) = \omega(C)$ for any $m \in \NN$. More broadly, we have the following result due to Thomason~\cite{Thomason_1989}.
\begin{lemma}\label{lemma:thomson_blow-up}
    For any simple graph $C$ of order $n$ and for $m \in \NN$ going to infinity, we have
    \begin{equation*}
        k_t(C[m]) = \frac{t! \, k_t(C)}{n^t} {mn \choose t} (1 + o(1))
    \end{equation*}
    as well as
    \begin{equation*}
        k_t(\overline{C[m]}) = \frac{\sum_{j=1}^t j! \, S(t,j) \, k_j(\overline{C})}{n^t} {mn \choose t} (1 + o(1))
    \end{equation*}
    where $S(t,j) = \sum_{i=0}^{j} (-1)^i {j \choose i} (j-i)^t / j!$ is the Stirling number of the second kind.
\end{lemma}
This is a direct consequence of the fact that the fraction of not necessarily injective homomorphic copies of cliques of size $t$ in a graph stays the same under blow-up. The number of homomorphic copies asymptotically matches the number of injective copies, accounting for the $1+o(1)$ term. This statement also readily generalizes to arbitrary graphs and not just cliques as well as to non-balanced blow-ups. From a more practical perspective, \cref{lemma:thomson_blow-up} gives an efficient way to compute an upper bound on $c_{s,t}$ from any graph $C$, as well as an upper bound on $g_{s,t}$ from any graph $C$ with $\omega(C) \leq t-1$, through the blow-up sequence $(C[m])_{m \in \NN}$. We will describe several different approaches that we employed in order to find constructions whose blow-up sequence give good upper bounds for the problems presented in the introduction.

\subsection{The discrete optimization problems}

\cref{lemma:thomson_blow-up} motivates considering the discrete optimization problem
\begin{equation}\label{eq:opt_problem}
    \argmin_{\bs \in \{0,1\}^N} c(\bs),
\end{equation}
where we are minimizing a cost $c: \{0, 1\}^N \to \RR$ over a discrete search space consisting of binary vectors $\{0, 1\}^N$ for some $N \in \NN$. In particular, we considered encoding both simple graphs as well as Cayley graphs through binary vectors, though many other combinatorial structures can expressed this way, see for example the recent work of Wagner~\cite{Wagner_2020}.

\paragraph{The graph search space.} When constructing arbitrary graphs on $n$ vertices, a state $\bs \in \{0,1\}^N$ represents the edges of that graph, so that we have $N = {n \choose 2}$. We associate with each edge an entry $s_i$ in $\bs$ that indicates whether the edge is in the graph or not. Denoting the constructed graph by $C = C(\bs)$, the cost function is simply given by
\begin{equation}\label{eq:cost_function}
    c(\bs) = \lim_{m \to \infty}  k_s(\overline{C[m]}) +  \lambda \, k_t(C[m]) \, ,
\end{equation}
where $\lambda = 1$ in the case of $c_{s,t}$, except when searching for particular improvements on the bounds on $c_{s,t}(x)$, and $\lambda \gg 1$ in the case of $g_{s,t}$ to act as a Lagrangian multiplier to ensure that the constraint $k_t(C[m]) = 0$ is fulfilled. This cost function is easily calculated through \cref{lemma:thomson_blow-up}.

\paragraph{The Cayley graph search space.} The effectiveness of any search method will primarily be governed by the number $N$ of variables used to construct the graph. However, at least for the motivating $K_4$-Ramsey multiplicity problem, we found that the quality of a construction was strongly dependent on its number of vertices $n$. Since in the general graph space $N={n\choose 2}$ is quadratic in $n$, searching for constructions becomes intractable when the number of vertices reaches the fourties. In order to access graphs beyond that, we explored the family of Cayley graphs, for which more efficient description is available.

Given a group $\bf G$ and a set $S \subseteq {\bf G}\setminus \{1\}$ satisfying $S = S^{-1}$, the corresponding Cayley graph $C=C({\bf G}, S)$ has vertex set $\bf G$ and two vertices $g_1, g_2$ are connected by an edge if $g_1 s = g_2$ for some $s \in S$. A state $\bs \in \{0,1\}^N$ represents the generating set $\bigcup_{{\bf s}_A=1} A$ where each subset $A=\{g, g^{-1}\}$ for $g\in {\bf G}$ corresponds to an entry ${\bf s}_{A}$ in $\bs$.
Denoting the Cayley graph constructed this way by $C = C(\bs)$, the relevant cost function is again given by \cref{eq:cost_function}.

The number $N$ of variables needed to encode a Cayley graph is between $|{\bf G}|/2$ and $|{\bf G}| - 1$. This is only linear in the number of vertices, which allows us to search for graphs an order of magnitude larger.
Additionally, Cayley graphs form a substantial subset of vertex-transitive graphs~\cite{HoltRoyle_2020, HoltRoyleTracey_2021}, making them a highly relevant source of constructions. Both circumstantial evidence and the following lemma, stating that in an optimal construction, the number of $K_4$ and $\overline{K_4}$ at each vertex is asymptotically the same, support this relevance. Given a graph $G$ and set of vertices $S \subseteq V(G)$, let $k_t(G, S)$ denote the number of cliques of size $t$ in $G$ containing all vertices of $S$. The proof is based on that of \cite[Proposition 8]{balogh2017rainbow}.
\begin{lemma}
    For any $t \geq 3$ and sequence of graphs $(G_n)_{n \in \NN}$ of order $n$ satisfying $\lim_{n \to \infty} k_t(\overline{G_n}) + k_t(G_n) = c_t$, we have 
    \begin{equation*}
        \max_{u,v \in V(G_n)} \big| k_t(\overline{G_n},\{u\}) + k_t(G_n,\{u\}) - k_t(\overline{G_n},\{v\}) - k_t(G_n,\{v\}) \big| / {n-1 \choose t-1} = o\left( 1 \right).
    \end{equation*}
\end{lemma}
\begin{proof}
    Let us write
    \begin{equation*}
        \kappa(G, S) = \big( k_t(G, S) + k_t(\overline{G}, S) \big) / {n-|S| \choose t-|S|}
    \end{equation*}
    as well as $\kappa(G) = \kappa (G, \emptyset)$. Assume there exist vertices $u_n, v_n \in V(G_n)$ and $\varepsilon > 0$ satisfying $\kappa(G_n, u_n) - \kappa(G_n, v_n) \geq \varepsilon$ for all $n \in \NN$. Consider the sequence of graphs $G_n'$ obtained by deleting $u_n$ from $G_n$ and duplicating $v_n$ as $v_n'$ along with its relations, where we choose to include $v_nv_n'$ in $E(G_n')$. We have
    \begin{align*}
        \kappa(G'_n) - \kappa(G_n) & \leq \kappa(G_n,\{v_n\}) - \kappa(G,\{u_n\}) + o(1) \leq - \varepsilon + o(1),
    \end{align*}
    where the $o(1)$ term comes from all cliques containing both $v_n$ and $u_n$ or $v_n'$. This contradicts the assumption that $\lim_{n \to \infty} k_t(\overline{G}) + k_t(G) = c_t$.
\end{proof}

\subsection{Heuristic search methods.}

Since the early 80s, many heuristic methods have been suggested for solving NP-hard optimization problems, particularly for combinatorial problems with discrete search spaces. We focus on two well-established methods that gained traction in the search for Ramsey numbers~\cite{Exoo_1998, ExooTatarevic_2015, MckayRadziszowski_1997} and provide an accessible introduction to both, complementing a recent growing interest in computational means in Extremal Combinatorics~\cite{LidickyPfender_2017, Rowley_2022, Wagner_2020, Wagner_2021}. Although heuristics inherently lacking global optimality guarantees, matching flag algebra lower bounds and stability results can, in this particular case, establish the optimality and uniqueness of heuristic solutions.  For more information on the subject of heuristic search methods, we refer the interested reader to the handbook of Gendreau and Potvin~\cite{GendreauPotvinEtAl_2010}. We also discuss the efficacy of a more recent Machine Learning-based approach compared to these methods in \cref{sec:ml-approach}.

\paragraph{Simulated Annealing.} Simulated Annealing (SA) is a probabilistic technique that can be interpreted as a modified local search. It accepts worse states according to a probabilistic acceptance criterion, which is modified over time to reject worse states and avoid getting trapped in local minima. SA was originally proposed by Kirkpatrick, Gelatt and Vecchi~\cite{KirkpatrickGelattVecchi_1983}, and its impact has been significant; in 2014, the original paper was listed as one of the 100 most cited scientific papers of all time~\cite{top100}. \cref{alg:simulated_annealing} describes the algorithm in pseudocode.

\begin{algorithm}\setstretch{1.0}
\caption{Simulated Annealing}\label{alg:simulated_annealing}
\smallskip
\KwData{initial state $\bs_0 \in \{0, 1\}^N$, cost function $c: \{0, 1\}^N \to \mathbb{R}$, neighborhood function $\mN: \{0,1\}^N \to \mP(\{0,1\}^N)$, number of iterations $I \in \NN$, temperatures $(t_i)_{1 \leq i \leq I}$}
\KwResult{best state found $\bs^{\star} \gets \argmin_{\bs \in \{\bs_0, \bs_1, \ldots, \bs_I\}} c(\bs) \in \{0,1\}^N$}
\For{$i\gets1$ \KwTo $I$}{
    $\bs_c \gets $ uniform random sample from $\mN(\bs_{i-1})$\\
    \eIf{$\min\big(\exp ( (c(\bs_{i-1}) - c(\bs_c)) / t_i ), 1) \geq \mathrm{rand} (0, 1)$}
    {
        $\bs_i \gets \bs_c$\\
    }{
        $\bs_i \gets \bs_{i-1}$\\
    }
}
\end{algorithm}

To more precisely describe the algorithm, let  $\mN: \{0,1\}^N \to \mP(\{0,1\}^N)$ denote some notion of neighborhoods of the states. We restricted ourselves to considering states as neighboring if their Hamming distance is $1$. The algorithm starts with a random state $\bs_0$ and executes a fixed number of iterations $I$, where in each iteration $1 \leq i \leq I$ we pick a candidate state $\bs_c$ uniformly at random from $\mN(\bs_{i-1})$ and accept it based on the probabilistic Metropolis' criterion~\cite{metropolis1953equation}. See also Dueck and Scheuer~\cite{DueckScheuer_1990} for a variant that avoids the probabilistic nature of SA. The temperature sequence $(t_i)_{1 \leq i \leq I}$ typically converges to $0$, and SA functions more like a local search as it does. There are many details when implementing SA, and the presentation here should not be considered authoritative. Due to the relatively cheap computation of $c(\bs')$ for any $\bs' \in \mN (\bs)$, we implemented a variant of SA that avoids rejecting states by directly sampling candidates from an appropriate distribution dictated by Metropolis' criterion, as previously suggested by Greene and Supowit~\cite{GreeneSupowit_1986}.

\paragraph{Tabu Search.} Tabu Search, an even simpler search heuristic than SA, was suggested by Glover~\cite{Glover_1986, Glover_1989, Glover_1990}. Like SA, it can be viewed as a modified local search aimed at avoiding local optima and cycles. We start with a randomly initialized state $\bs_0$ and require a notion of neighborhood $\mN: \{0,1\}^N \to \mP(\{0,1\}^N)$. We execute $I$ iterations, where in each iteration $1 \leq i \leq I$, we pick the neighboring state $\bs \in \mN(\bs_{i-1})$ with the lowest associated cost and that has not been visited recently, regardless of whether it improves upon $c(\bs_{i-1})$. 
\cref{alg:tabu_search} contains the pseudo-code for Tabu Search.

\begin{algorithm}\setstretch{1.0}
\caption{Tabu Search}\label{alg:tabu_search}
\smallskip
\KwData{initial state $\bs_0 \in \{0, 1\}^N$, cost function $c: \{0, 1\}^N \to \mathbb{R}$, neighborhood function $\mN: \{0,1\}^N \to \mP(\{0,1\}^N)$, history length $\ell \in \NN$, number of iterations $I \in \NN$, }
\KwResult{best state found $\bs^{\star} \gets \argmin_{\bs \in \{\bs_0, \bs_1, \ldots, \bs_I\}} c(\bs) \in \{0, 1\}^N$}
\For{$i \gets1$ \KwTo $I$}{
    $\bs_i \gets  \argmin_{\bs \in \mN(\bs_{i-1}) \setminus \{\bs_{i-\ell}, \ldots, \bs_{i-1}\}} c(\bs)$\\
}
\end{algorithm}

There are many degrees of freedom when implementing this algorithm. One modification we made was to the history implementation: instead of storing a history of the last $\ell$ states and excluding those from the update, we store a list of the last $\ell$ modified bits and exclude any state that differs from the current one in one of those bits. This slightly increases the number of excluded states but drastically reduces the computational and implementation effort required to determine which states to exclude.

\medskip

It is crucial to have an efficient implementation of the cost function $c$ for both methods, particularly since they are often run in parallel for various initial states. In our application, we precomputed the relevant indices for all cliques, considering multiplicity, and stored the results in a matrix format. This approach allowed us to evaluate $c$ using only elementary matrix operations on a GPU, substantially improving efficiency when assessing multiple states in parallel.

\subsection{Constructions} \label{sec:constructions}

We implemented all search methods in {\tt Python} and {\tt Pytorch} and logged the results using {\tt Weights \& Biases}~\cite{wandb}. We relied on the {\tt GAP}~\cite{gap} component in {\tt SageMath} and in particular the {\tt Small Groups library} for Cayley graph constructions. We will represent graphs using their {\tt graph6} representation, a 6-bit encoding of the upper triangle of the graph adjacency matrix. This representation can be most easily decoded using {\tt SageMath} and a formal description is available at \href{http://cs.anu.edu.au/~bdm/data/formats.txt}{\nolinkurl{cs.anu.edu.au/~bdm/data/formats.txt}}. Graph descriptions too large to be included in this paper are available at \href{https://doi.org/10.5281/zenodo.6364588}{\nolinkurl{doi.org/10.5281/zenodo.6364588}}. We also provide a survey of the derivation of previous constructions and describe the process we used to obtain our bounds for additional context.

\subsubsection{Ramsey multiplicity -- \cref{thm:4-ramsey_multiplicity} and \cref{thm:5-ramsey_multiplicity}} \label{sec:proof_ramsey_multiplicity}

\paragraph{Previous constructions.}
The original counterexample by Thomason~\cite{Thomason_1989} to Erd\H{o}s' conjecture for $t \ge 4$ was given by the blow-up sequence of graphs formed by vectors in orthogonal geometries.
This construction gives $c_4 < 0.03050$ and $c_5 < 0.001769$ and Thomason improved the bound for $t = 4$ to $c_4 < 0.03030$ through a computer based local search around its two-fold blow-up. Franek and R\"odl~\cite{FranekRodl_1993} presented a class of more simply describable Cayley graphs containing a construction on $2^{10} = 1024$ vertices, also found through a computer search, whose blow-up sequence likewise disproves Erd\H{o}s' conjecture for $t=4$ by giving an upper bound of $c_4 < 0.03052$.
This was generalised by Franek~\cite{Franek_2002} and Deza, Franek, and Liu~\cite{deza2012conjecture}, giving the currently best known bounds for $c_6$, $c_7$, and $c_8$.

Thomason~\cite{Thomason_1997} slightly improved upon his constructions for $t \in \{4,5\}$ by noticing that the XOR graph product $\otimes$\footnote{Given two graphs $G_1 = (V_1, E_1)$ and $G_2 = (V_2, E_2)$, their XOR graph product $G_1 \otimes G_2$ has vertex set $V_1 \times V_2$ and two vertices $(v_1, v_2), (v_1', v_2') \in V_1 \times V_2$ are connect if and only if \emph{either} $v_1 v_1' \in E_1$ and $v_2 v_2' \notin E_2$ \emph{or} $v_1 v_1' \notin E_1$ and $v_2 v_2' \in E_2$. Here we consider loops in the graph, that is in particular $v v \notin E_i$ for $v \in V_i$ unless $G_i$ has  a loop at $v$ for $i \in \{1,2\}$. Thomason calls this graph product the \emph{tensor product} and Wolf~\cite{Wolf_2010} states that it is also known as the \emph{Cartesian product}. It seems however that the tensor product is more commonly used for the case $v_1 v_1' \in E_1$ and $v_2 v_2' \in E_2$ and the Cartesian product for the case either $v_1 v_1' \in E_1$ and $v_2 = v_2'$ or $v_1 = v_1'$ and $v_2 v_2' \in E_2$, so to avoid confusion we use the unambiguous terminology of Alon and Lubetzky~\cite{AlonLubetzky_2007}. We will also write $G^{\otimes k}$ for the $k$-fold XOR graph product of a graph $G$.} has some favorable properties for this problem, which were previously already observed by Chung and Graham~\cite{ChungGraham_1991}. In particular, by simply computing a particular vector of graph densities for two given graphs $G_1, G_2$, one can determine the corresponding vector of $G_1 \otimes G_2$ as the element-wise product of those two vectors from which one can then easily determine the relevant upper bound on $c_{s,t}$ through a vector product with some appropriate weight vector. Computationally, this gives one the opportunity to pre-calculate these density vectors for small or medium sized graphs and then cheaply compute the upper bounds given by even very large graph products. The best upper bound found this way for $c_4$ is given by the blow-up sequence of $K_4 \otimes M_4 \otimes G_{18}$, where $M_4$ denotes a perfect matching of order $4$ and $G_{18}$ the complement of $K_3^{\otimes 2}\otimes \overline{K_2}$, giving roughly $c_4 < 0.03029$, and the best upper bound for $c_5<0.001720$ is given by the blow-up sequence of $K_3 \otimes M_4^{\otimes 3}$. Thomason also observed that his original constructions from \cite{Thomason_1989} can be described as $K_4 \otimes M_4^{\otimes (t-1)}$ for $t \ge 5$ and as a subgraph of this for $t=4$.
Note that the XOR graph product in some sense works as a generalisation of the blow-up, since $G[m] = G \otimes \overline{K_m}$ for a loopless graph $G$.

The improvement of Even-Zohar and Linial~\cite{even2015note} was based on the observation that $G_{18}$ can also be seen as a blow-up of $K_2$ with a copy of $K_3 \otimes K_3$ inserted at every vertex, rather than an independent set.
This motivates the \emph{composition} $G \odot H$ of two graphs $G$ and $H$, where a copy of $G$ is placed at every vertex of $H$\footnote{Formally, given two graphs $G_1 = (V_1, E_1)$ and $G_2 = (V_2, E_2)$, their composition $G_1 \odot G_2$ has vertex set $V_1 \times V_2$ and two vertices $(v_1, v_2), (v_1', v_2') \in V_1 \times V_2$ are connect if and only if \emph{either} $v_1 v_1' \in E_1$ \emph{or} $v_1=v_1'$ and $v_2 v_2' \in E_2$.}, e.g.~$G_{18}=(K_3 \otimes K_3) \odot K_2$.
Their construction is $K_4 \otimes M_4 \otimes (K_3 \otimes K_3)^{\odot n}$ for $n$ tending to infinity, where $G^{\odot n}$ is the $n$-fold composition of $G$ with itself.
This is the only example of a construction that is not a standard blow-up and it remains unclear what the role of $K_4 \otimes M_4$ is.
While a similar decomposition followed by replacing some part with an $n$-fold composition could possibly enhance other bounds, it is computationally challenging to identify a suitable decomposition with XOR products for larger graphs, if such a decomposition even exists.

We finally note that we pushed both the approaches of Franek and R\"odl in~\cite{FranekRodl_1993}, Thomason in~\cite{Thomason_1997}, and Even-Zohar and Linial~\cite{even2015note} to what we found feasible using our capabilities and computational resources. We found no or only the most marginal of improvements, implying that these approaches have probably exhausted their full potential.

\paragraph{Constructions in \cref{thm:4-ramsey_multiplicity} and \cref{thm:5-ramsey_multiplicity}.} 

We ran heuristic searches to construct graphs on up to $40$ vertices that minimize the cost function given by \cref{eq:cost_function}. The smallest graph we found whose blow-up sequence establishes a value below $1/32$ was of order $33$, giving a value of $0.03118$. The smallest previously known such graph was described by Thomason~\cite{Thomason_1989} and was of order $36$. We also found a graph on $32$ vertices whose \emph{weighted} blow-up gives a value slightly below $1/32$. We found no further constructions on less than $40$ vertices giving any significantly better values.

Given that the blow-up sequences of small graphs seem to yield little of interest, we considered Cayley graphs next.
We have already noted that the best construction given by a blow-up sequence ($c_4 < 0.03029$) is based on a Cayley graph in $\ZZ_3^{\times 2} \times \ZZ_2^{\times 5}$ (order $288$) and a Tabu Search in that particular group yielded no improvement.
However, already in $\ZZ_3 \times \ZZ_2^{\times 6}$ (order $192$) we were able to find a graph whose blow-up sequence even improves upon the previous best value found by Even-Zohar and Linial ($c_4 < 0.3028$), giving an upper bound of $c_4 < 0.03027$.
Going to $\ZZ_3 \times \ZZ_2^{\times 8}$ (order $384$) allows us to already achieve a bound if $c_4 < 0.03015$ and going up to $\ZZ_3 \times \ZZ_2^{\times 8}$ (order $768$) produced the graph whose blow-up sequence gives the upper bound  $c_4 \leq 4551721 \cdot 2^{-24} \cdot 3^{-2} < 0.03015$ stated in \cref{thm:4-ramsey_multiplicity}.

We note that there seems to be no particular significance to the fact that we derived these constructions in groups defined only through direct products.
In fact, running the Tabu Search on all groups of order at most 192 revealed (1) that in general decent constructions seem to be found in groups of order $3 \cdot 2^n$ and (2) many groups of that order perform significantly better than the group $\ZZ_3 \times \ZZ_2^{\times n}$.
For example, for groups of order 192 the best value we found was around $0.03021$.
Unfortunately we were unable to determine any patterns indicating which groups might be preferable when going to groups of order 384 or 768.
The sheer number of these groups and the amount of cliques to consider unfortunately makes it impossible to run a Tabu Search on anything more than a small selection of them.

Regarding the value for $c_5$, the previous best construction can be described as the blow-up sequence of a Cayley graph in $\ZZ_3 \times \ZZ_2^{\times 6}$ (order $192$). A search run on Cayley graphs in this group yielded the slight improvement presented in \cref{thm:5-ramsey_multiplicity}. Due to the fact that we need to consider cliques of size $5$, we were unable to go up Cayley graphs of order $384$ or $768$ as we did for $c_4$.

\subsubsection{Bounded clique number -- \cref{thm:axis_ramsey_multiplicity}} \label{sec:proof_axis_ramsey_multiplicity}

\paragraph{Previous constructions.} Complete graphs are obvious candidates to consider as constructions in this context, as the Tur{\'a}n graph $T_{t-1}(n)$ can be described as $K_{t-1}[n/(t-1)]$ whenever $t-1$ divides $n$, that is the blow-up sequences of $K_{t-1}$ give tight upper bounds for $g_{2,t}$ for arbitrary $t \geq 2$. The same holds for $g_{3,3}$ as the blow-up sequence of $K_2$ gives an upper bound of $1/4$ and $g_{3,3} \geq c_3 = 1/4$ by~\cite{Goodman_1959}. Das et al.~\cite{DasEtAl_2013} showed that the upper bound given by the blow-up sequence of $K_3$ for $g_{3,4}$ is tight and the results of Pikhurko and Vaughan~\cite{PikhurkoVaughan_2013} imply that the same is true for $K_{t-1}$ and $g_{3,t}$ for $t \in \{5,6,7\}$.

However, Nikiforov~\cite{Nikiforov_2001} showed this can only cover finitely many cases of $g_{s,t}$ when $s,t \geq 3$ and for the Ramsey multiplicity problem, the blow-up sequences of complete graphs are in general far from even the performance of random graphs. Das et al.~\cite{DasEtAl_2013} showed that for $g_{4,3}$ the tight upper bound is given by the blow-up sequence of $C_5$ and Pikhurko and Vaughan~\cite{PikhurkoVaughan_2013} proved the same for $g_{5,3}$ and that the the blow-up sequence of the Clebsch graph on $16$ vertices establishes tight upper upper bounds for $g_{6,3}$ and $g_{7,3}$. Note that both $C_5$ and the Clebsch graph are vertex-transitive and in fact can be realized as Cayley graphs, which of course is also true of complete graphs.

Pikhurko and Vaughan~\cite{PikhurkoVaughan_2013} also found a construction for $g_{4, 4}$ that relies on the weighted blow-up sequence of a non-vertex-transitive graph of order $8$ where the weights of the blow-up align with the two vertex orbits of size $4$ of the graph. This graph is in fact one of three $(3,4)$-Ramsey graphs of order $9$ and has the {\tt graph6} representation {\tt GK\^{}d\}w}. The upper bound of $(14 \cdot2^{1/3} - 11)/2^7 \approx 0.034578$ given by this construction seemingly aligns with the lower bound indicated by the flag algebra approach, but they were unable to turn this into a rigorous proof.

\paragraph{Constructions in \cref{thm:axis_ramsey_multiplicity}.}
The construction used to establish $g_{4,5}$ is given by the blow-up sequence of the unique vertex-transitive graph of order $13$, degree $8$, clique number $4$ and independence number $2$, referred to as $C_{R(3,5)}$ in the introduction. This is the only $(3,5)$-Ramsey graph of order $13$ and has the {\tt graph6} representation 
\begin{center}\begin{verbatim}
    LJ]lmZRnn]]\v[
\end{verbatim}\end{center}
This construction can be easily found using a search heuristic in the graph space or through an exhaustive search of all Cayley graphs of $\ZZ_{13}$.

\paragraph{Additional constructions and bounds.}
Noting that upper bounds to both $g_{4,4}$ and $g_{5,4}$ are established through Ramsey graphs, we searched McKay's collection\footnote{Brendan McKay has created a collection of combinatorially interesting constructions, among them several complete and partial lists of Ramsey graphs, which is available at \href{https://users.cecs.anu.edu.au/~bdm/data}{\nolinkurl{users.cecs.anu.edu.au/~bdm/data}}.} of such graphs for additional constructions whose weighted blow-up sequence gives good upper bounds for $g_{s,t}$. In particular, we considered the $(3,6)$-, maximal $(3,7)$-, $4$-, $(4,5)$-, $(4,6)$- and $5$-Ramsey graphs on respectively $17$, $22$, $17$, $24$, $35$, and $42$ vertices. Of these there are respectively $7$, $22$, $1$, $352 \, 366$, $37$, and $328$ graphs. We were able to derive the following bounds:
\begin{align*}
    0.008175 & < g_{5, 4} \leq 0.008584, \\ % 
    0.002020 & < g_{5, 5} \leq 0.002136, \\ %  
    0.006406 & < g_{4, 6} \leq 0.006773,\\ %
    0.003275 & < g_{4, 7} \leq 0.003637,\\ %
    0.0006319 & < g_{5, 6} \leq 0.0008433 \text{ and} \\ %
    0.0001978 & < g_{5, 7} \leq 0.0003500. %
\end{align*}
All lower bounds were established using the flag algebra approach, see \cref{sec:lower_bounds}.
The construction used to establish the upper bound for $g_{5,4}$ is given by the weighted blow-up sequence of a $(5,4)$-Ramsey graph of order $13$ with $5$ vertex orbits which has the {\tt graph6} representation
\begin{center}\begin{verbatim}
    L@OZ@\Vmmu}hzL
\end{verbatim}\end{center}
and was found as a subgraph of one of the $(5,4)$-Ramsey graphs of order $24$ with $24$ vertex orbits.
%W_K_pWG[d@ZGtKp_AxpbrDNA\b\kUQVTD@tZwuBR\JotvkC

The construction used to establish the upper bound for $g_{5,5}$ is the same used to establish $g_{4,5}$, that is it is given by the blow-up sequence of the unique $(3,5)$-Ramsey graph of order $13$ which has the {\tt graph6} representation 
\begin{center}\begin{verbatim}
     LJ]lmZRnn]]\v[
\end{verbatim}\end{center}
%WQ`T\]\Uh{kxXhrOURZzok\Zw\Gz]YUmw_~zra{}nfB^Cur
%
The construction used to establish the upper bounds for $g_{4,6}$ and $g_{5,6}$ is given by the weighted blow-up sequence of a $(3,6)$-Ramsey graphs of order $17$ which has $9$ vertex orbits and the {\tt graph6} representation
\begin{center}\begin{verbatim}
    P~TktL|vdu{{^]vl[z|v]B~{
\end{verbatim}\end{center}
The construction used to establish the upper bounds for $g_{4,7}$ and $g_{5,7}$ is given by the weighted blow-up sequence of a $(3,7)$-Ramsey graphs of order $22$ which has $11$ vertex orbits and the {\tt graph6} representation
\begin{center}\begin{verbatim}
    U`K~vj\zff\Zt]rlzv^Zm}z^v]r~^r}~m}~kn^vG
\end{verbatim}\end{center}
We also checked a library of small vertex-transitive graphs~\cite{HoltRoyle_2020} but found no additional constructions improving upon the upper bounds. For $g_{5,4}$ and $g_{5,5}$ we additionally ran search heuristics both in the graph and Cayley graph space and also found no improvements.

\subsubsection{Off-diagonal Ramsey multiplicity -- \cref{thm:offdiagonal_ramsey_multiplicity}} \label{sec:proof_offdiagonal_ramsey_multiplicity}

The construction used to establish the upper bound for $c_{3,4}$ is given by the blow-up sequence of the Schl\"afli graph, a vertex-transitive graph of order $27$, with the {\tt graph6} representation 
\begin{center}\begin{verbatim}
    ZBXzz|z^Z|tFixjTtp|mFk\uqm|gz}]FbHvHqjh]WzFy[RmtSUztaLvyF`vw
\end{verbatim}\end{center}
The construction used to establish $c_{3,5}$ is better described as a construction for $c_{5,3}$ in which case it is given by the blow-up sequence of the complement of the Schl\"afli graph, that is the graph with the {\tt graph6} representation 
\begin{center}\begin{verbatim}
    Z??G`@?@wrDSLGQoigbKO]CA?^{VDsjIqehgmK[EM[OzIqCyegO|FO_^{?_?
\end{verbatim}\end{center}
Both of these constructions can be found through search heuristics of graphs of order $27$. The construction used to establish the upper bound for $c_{4,5}$ is given by the blow-up sequence of vertex-transitive graph on $128$ vertices. It was found using a search of Cayley graphs of order at most $128$.

\paragraph{Additional constructions in \cref{fig:c_34}.} The graph on $40$ vertices is vertex-transitive, has degree $17$, clique number $3$ and independence number $12$. %g????@~^r}Nw^o^oNwB}??@??G??_?@??O??G??A???ON~~_n~}`N~}PF~~C@~~o_N~}A?~~wC@~~oF}?~oB~?nw?~oR}?F}C^o?^w@~?O~oB}@?~oB}A?^w@~A?C@~_?B~
The graph on $24$ vertices is vertex-transitive, has degree $11$, clique number $3$ and independence number $6$ and has the {\tt graph6} representation
\begin{center}\begin{verbatim}
    W@TBOkkJBBAoSCW?Qv{V}jRrhfC{UEfaRPtAw\_ckqGt`oL
\end{verbatim}\end{center}
The graph on $128$ vertices is vertex-transitive, has degree $78$, clique number $5$ and independence number $16$.
Lastly, the graph on $112$ vertices is vertex-transitive, has degree $68$, clique number $5$ and independence number $14$.
The two smaller graphs were found using a vertex-transitive graph library~\cite{HoltRoyle_2020, HoltRoyleTracey_2021} and the two larger using a search of Cayley graphs of order at most 128 where the $\lambda$ parameter was adjusted accordingly in \cref{eq:cost_function}.

\section{Lower bounds and stability through flag algebras} \label{sec:lower_bounds}

Razborov~\cite{Razborov_2007} proposed using finite model theory to describe the algebraic structure underlying many techniques in Extremal Combinatorics.
This approach allows one to derive lower bounds for problems like those studied in this paper through a semi-definite program (SDP).
This method has been widely used over the past decade and there exist not only several very good introductions to the topic~\cite{FalgasVaughan_2013, PikhurkoVaughan_2013, SilvaEtAl_2016} but also a tool in the form of {\tt flagmatic}~\cite{flagmatic}.

If the lower bound obtained with this approach matches a constructive upper bound it is sometimes possible to additionally infer the uniqueness of the construction from the flag algebra based proof.
We improve the toolset developed in~\cite{PikhurkoEtAl2019, PikhurkoVaughan_2013} used to derive such stability results in order to show that several of the constructive bounds found using search heuristics in fact represent global optima.
We start with a summary of the theory behind the approach based on~\cite{PikhurkoEtAl2019, PikhurkoVaughan_2013}, then turn to the stability aspect, and finally go into some detail about the practical aspects of how our proofs can be verified.

\subsection{The theory behind the flag algebra approach}\label{sec:flag_algebra_theory}

Suppose we have a (possibly empty) family $\mX$ of \emph{forbidden} graphs. Let $\mG_n = \mG_n(\mX)$ denote set of all graphs up to isomorphism of some given order $n$ not containing any graph in $\mX$ as an induced subgraph and write $\mG = \mG(\mX) = \bigcup_{n \in \NN} \mG_n (\mX)$. For two graphs $G$ and $H$ we let $d_H(G)$ denote the \emph{induced subgraph density} of $H$ in $G$, that is the probability that $|V(H)|$ vertices chosen uniformly at random in $G$ induce a copy of $H$. Consider a graph parameter $\lambda: \mG \to \RR$ for which there exists some $n_0 \in \NN$ such that $\lambda$ satisfies the averaging equality
\begin{equation} \label{eq:averaging_equality}
    \lambda(G) = \sum_{H \in \mG_n} d_H(G) \, \lambda(H)
\end{equation}
for any $n \geq n_0$ and $G \in \mG$ of order at least $n$. We note that the results in this paper are limited to families of forbidden graphs $\mX = \emptyset$ and $\mX = \{K_t\}$ as well as the graph parameter $\lambda(G) = d_{K_s}(\overline{G}) + d_{K_t}(G)$, which clearly satisfies \cref{eq:averaging_equality} with $n_0 = \max\{s,t\}$.

For any infinite subset $\mH \subseteq \mG$, we write
\begin{equation*}
    \lambda(n, \mH) = \min_{G \in \mH \cap \mG_n} \lambda(G) \quad \text{and} \quad  \lambda(\mH) = \liminf_{n \to \infty} \lambda(n, \mH)
\end{equation*}
where for notational convenience $\lambda(n, \mH) = \infty$ when $\mH \cap \mG_n = \emptyset$. We are of course primarily interested in studying $\lambda(\mG)$, though we keep the notation general, as we will need it when establishing stability. Note that a trivial lower bound that follows from \cref{eq:averaging_equality} is  $\lambda(\mG) \geq \lambda(m, \mG)$ for any arbitrary $m \geq n_0$.
\begin{example}
    In the case of the asymptotic version of Goodman's result, the trivial lower bound only gives us $c_3 \geq 0$ when $m \in \{3, 4, 5\}$ since $R(3, 3) = 6$ and, respectively, $1/10$, $4/35$, $1/7$, $1/7$, and $1/6$ for $m \in \{6,7,8,9,10\}$, a far cry from the true value of $1/4$.
\end{example}
The goal of the flag algebra approach is to establish the existence of coefficents $a_H \in \RR$ for some fixed $m \in \NN$ satisfying the inequality
\begin{equation} \label{eq:inequalities}
    \sum_{H \in \mG_m} d_H(G) \, a_H + o(1) \geq 0
\end{equation}
for any $G$ of order at least $m$, as this implies the, hopefully improved, lower bound of 
\begin{equation} \label{eq:improved_lower_bound}
    \lambda(\mG) \geq \min_{H \in \mG_m} \{ \lambda(H) - a_H \}.
\end{equation}
In order to establish the type of coefficients commonly achieved through the solution of an SDP, we will need some additional notation. Note that the term \emph{strong (graph) homomorphism} between two (potentially looped) graphs $G_1$ and $G_2$ refers to any map $\psi : V(G_1) \to V(G_2)$ satisfying $\{v_1, v_2\} \in E(G_1)$ if and only if $\{\psi(v_1), \psi(v_2)\} \in E(G_2)$.

\begin{definition}
    A \emph{type} $\tau = (T, \varphi)$ consists of a graph $T \in \mG$ of order $v$ that is fully and distinctly labelled through $\varphi: [v] \hookrightarrow V(T)$. Note that we can have $v = 0$, in which case $\tau$ is  the \emph{empty type} $\varnothing$. A \emph{$\tau$-flag} $(F, \psi)$ consists of a graph $F \in \mG$ of order at least $v$ that is a partially labelled through the injective map $\psi: [v] \hookrightarrow V(F)$ that also satisfies that $\psi \circ \varphi^{-1}: V(T) \hookrightarrow V(F)$ defines an injective strong homomorphism of $T$ into $F$.
\end{definition}
We will denote the set of all $\tau$-flags of order $l \geq v$ by $\mF_{\tau}^{l}$.
Let $F, F' \in \mF_{\tau}^{l}$ be two $\tau$-flags of same order and $H \in \mG$ an arbitrary graph of order at least $2l - v$. Let $\theta:[v] \hookrightarrow V(H)$ be an arbitrary injective map implying a partial labelling of $H$ (but not necessarily turning $(H, \theta)$ into a $\tau$-flag) and write $d^{\theta}_F(H)$ for the probability that $l-v$ vertices selected uniformly at random in $V(H) \setminus \theta([v])$ together with the vertices labelled by $\theta$ induce a flag isomorphic to $F$. Obviously this value is $0$ whenever $(H, \theta)$ is not a $\tau$-flag. We will write $d_F(H) = \mathbb{E}_{\theta} \, d^{\theta}_F(H)$ for the \emph{flag density}, where we are taking the uniform distribution over all possible injective maps $\theta$. We will likewise write $d^{\theta}_{F, F'}(H)$ for the probability that a subset $S_1$ of $V(H) \setminus \theta([v])$ of size $l-v$ chosen uniformly at random together with the vertices labelled through $\theta$ is isomorphic to $F$ and that another subset $S_2$ of $V(H) \setminus \big( \theta([v]) \cup S_1 \big)$ of size $l-v$ chosen uniformly at random together with the vertices labelled through $\theta$ is isomorphic to $F'$. We will write $d_{F,F'}(H) =  \mathbb{E}_{\theta} \, d^{\theta}_{F, F'}(H)$ for the \emph{flag pair density}.

\begin{example}
    There is exactly one type $\tau$ of order $1$, i.e., that based on a graph with a single labelled vertex. There are also exactly two $\tau$-flags of order $2$, that consisting of an edge with a labelled vertex and that consisting of two isolated vertices with one labelled, and six $\tau$-flags of order $3$.
\end{example}

We note that \cref{eq:averaging_equality} holds for the flag pair density when $m \geq 2l-v$, that is 
\begin{equation}\label{eq:obs2}
    d_{F, F'}(G) = \sum_{H \in \mG_m} d_H(G) \, d_{F, F'}(H)
\end{equation}
for arbitrary $G \in \mG$. We also note that $d_F^{\theta}(G) \, d_{F'}^{\theta}(G) = d_{F, F'}^{\theta}(G) + O(1/n)$ and hence
\begin{equation}\label{eq:obs1}
    \mathbb{E}_{\theta} d_F^{\theta}(G) \, d_{F'}^{\theta}(G) = d_{F, F'}(G) + O(1/n).
\end{equation}
Using this notation, we can now state the heart of the SDP-based flag algebra approach.

\begin{proposition} \label{thm:flag_algebra}    For any integer $m \in \NN$, types $\tau_i$ of order $0 \leq v_i \leq m-2$ satisfying $v_i \equiv m \mod 2$, and positive semi-definite matrices $\mQ^{(i)}$ of size $|\mF_{\tau_i}^{l_i}| \times |\mF_{\tau_i}^{l_i}|$ with $l_i = (m + v_i)/2$ for $1 \leq i \leq r$, where we will use flags $F \in \mF_{\tau_i}^{l_i}$ as indices for $\mQ^{(i)}$, we have 
    \begin{equation}\label{eq:sdp_flag_algebra}
        \lambda(\mG) \geq \min_{H \in \mG_m} \left( \lambda(H) - \sum_{i=1}^r \sum_{F,F' \in \mF_{\tau_i}^{l_i}} \mQ^{(i)}_{F,F'} \, d_{F,F'}(H) \right).
    \end{equation}
\end{proposition}

\begin{proof}
    Let us establish that 
    \begin{equation}\label{eq:psd_consequence}
       \sum_{H \in \mG_m} d_H(G) \left(  \sum_{i = 1}^r  \sum_{F,F'\in \mF_{\tau_i}^{l_i}} \mQ^{(i)}_{F, F'} \, d_{F,F'}(H)  \right) \geq O(1/n)
    \end{equation}
    for any graph $G$ of order $n$. By rearranging and then applying both \cref{eq:obs2} and \cref{eq:obs1}, we have
    \begin{align*}
        \sum_{H \in \mG_m} d_H(G) \left( \sum_{F,F'\in \mF_{\tau_i}^{l_i}} \mQ^{(i)}_{F, F'} \, d_{F,F'}(H)  \right) = \mathbb{E}_{\theta} \!\!\! \sum_{F,F' \in \mF_{\tau_i}^{l_i}} \!\!\! \mQ^{(i)}_{F, F'} \, d_{F}^{\theta}(G) \, d_{F'}^{\theta}(G) + O(1/n).
    \end{align*}
    for all $1 \leq i \leq r$. \cref{eq:psd_consequence} therefore follows from the fact that the $\mQ^{(i)}$ are positive semi-definite and by summing over $1 \leq i \leq r$. By \cref{eq:psd_consequence} the terms 
    \begin{equation*}
       \sum_{i = 1}^r  \sum_{F,F'\in \mF_{\tau_i}^{l_i}} \mQ^{(i)}_{F, F'} \, d_{F,F'}(H)
    \end{equation*}
    can serve as $a_H$ in \cref{eq:inequalities}, establishing the claim of through \cref{eq:improved_lower_bound}.
\end{proof}

This proposition motivates the following SDP formulation: given two symmetric matrices $A$ and $B$ of equal size, we will use the inner product $\langle A,B \rangle = \textrm{tr}(A^T B)$. Write $D^{(i)}(H)$ for the symmetric matrix of size $|\mF_{\tau_i}^{l_i}| \times |\mF_{\tau_i}^{l_i}|$ for $1 \leq i \leq r$ whose entries are given by $d_{F,F'}(H)$ when using flags $F \in \mF_{\tau_i}^{l_i}$ as indices. Write $D_H$ for the symmetric block diagonal matrix with the integer $1$ as its first and $D^{(i)}(H)$ as the following blocks for $1 \leq i \leq r$. 
Finally, let $C$ denote the matrix of equal size to $D_H$ with all-zero entries except for the first, which is $1$.  We are now interested in solving 
\begin{align*} \label{eq:sdp}
    \max_{X \succeq 0} & \quad \langle C, X \rangle \\
    \text{subject to} & \quad \langle D_H, X \rangle \leq \lambda(H) \quad \text{for all } H \in \mG_m. \nonumber
\end{align*}
Finding a positive semi-definite matrices $X$ that solves this SDP is equivalent to finding $\mQ^{(i)}$ that maximize the right hand side of \cref{eq:sdp_flag_algebra}. Of course, most actual solvers will take advantage of the block-diagonal structure that we may assume for $X$ given the problem formulation, rather than optimizing over the whole space of positive semi-definite matrices.

\begin{example}
    For our toy example of $c_3$ let us choose $r = 1$ and suppress the  index $i$. We let $\tau$ be the only type on one vertex and set $l=2$ and $m=3$. Then $\mF_{\tau}^{2}$ consists of the two graphs on two vertices, $j$ edges and one labelled vertex, where $0 \leq j \leq 1$. Likewise, $\mH_3$ consists of the four graphs $H_j$ on $3$ vertices with exactly $j$ edges, where $0 \leq j \leq 3$. We only require one positive-semidefinite matrix of size $2 \times 2$ here and write it as
    \begin{equation*}
        \mQ = \begin{pmatrix}
        a & c\\
        c & b
        \end{pmatrix}.
    \end{equation*}
    Note that $\mQ$ is positive-semidefinite if and only if $a \geq 0$ and $ab - c^2 \geq 0$. By \cref{thm:flag_algebra} we need to maximize the minimum of the four expressions
    \begin{align*}
        \text{(i)} & \quad 1 - a, \\
        \text{(ii)} & \quad 0 - (a + 2c)/3, \\
        \text{(iii)} & \quad 0 - (b + 2c)/3 \text{ and} \\
        \text{(iv)} & \quad 1 - b.
    \end{align*}
    It is easy to see that this minimum is attained when $a = b = 3/4$ and $c = - 3/4$, in which case all four expressions attain the value of $1/4$.
\end{example}

\subsection{Establishing stability and exactness}\label{sec:stability_theory} 

After finding matching upper and lower bounds, it is natural to ask whether the construction is unique and stability holds, i.e., if anything that comes close to the optimal value must be close to the extremal construction.
Statements like this can usually be derived by extracting additional information from a flag algebra based proof and appealing to the Induced Removal Lemma of Alon et al.~\cite{AlonEtAl2000}.
Pikhurko et al.~\cite{PikhurkoEtAl2019} formalized this process, establishing sufficient criteria for various stability forms. We will present this succinctly while highlighting two improvements: introducing and strengthening the notion of \emph{reconstructors} to derive stability from smaller Flag-Algebra-based proofs, and formalizing an argument for establishing the optimal weighting of a blow-up. 
Let us start by establishing some additional notions relating to \cref{thm:flag_algebra}. 
\begin{definition}
    We call $(m, (\tau_i)_{i \in [r]}, (\mQ^{(i)})_{i \in [r]} )$ a \emph{certificate} of the bound given by \cref{thm:flag_algebra}. Assuming a certificate establishes equality in \cref{eq:sdp_flag_algebra}, that is it is a certificate of $\lambda(\mG)$, we call a graph $H \in \mG_m$ \emph{sharp} in it if $\lambda(H) - \sum_{i=1}^r \langle \mQ^{(i)}, D^{(i)}(H) \rangle = \lambda(\mG)$.
\end{definition}
A graph that is not sharp asymptotically does not appear as a subgraph in any extremal sequence, see for example Lemma 4.1 in~\cite{PikhurkoVaughan_2013} for a proof of this statement.

\begin{lemma}
    Let $(m, (\tau_i), (\mQ^{(i)}) )$ be a certificate of $\lambda(\mG)$. If $H \in \mG_m$ is not sharp, then for any sequence $(G_n)_{n \in \NN}$ of graphs in $\mH$ of increasing order that satisfies $\lim_{n \to \infty} \lambda(G_n) = \lambda(\mG)$ we have $d(H, G_n) = o(1)$.
\end{lemma}

Sharp graphs therefore usually give a good indication of which graphs can occur as subgraphs in an extremal sequence and have in the past been used to find novel constructions.
\begin{example}
    Briefly returning to our toy example of $c_3$, we had a certificate with $m = 3$ in which all graphs were sharp, indicating that there might be (and in fact are) extremal sequnces for $c_3$ containing any subgraph of order $3$ with positive density.
\end{example}
If a certificate $(m, (\tau_i), (\mQ^{(i)}) )$ is sufficiently large in relation to our believed extremal construction $C$, say $m \geq |V(C)| + 1$, and stability is believed to hold, then establishing it is often a reasonably straight-forward matter of verifying that the only sharp graphs are those that occur with positive density in the blow-up sequence of $C$ and then invoking the Induced Removal Lemma. This is often far beyond the realm of being computationally feasible, so we need a way to establish the structure of a construction $C$ using only graphs of order $m$, where $m$ is preferably as small as possible. This was already done implicitly in~\cite{PikhurkoVaughan_2013} and the same ideas can also be found in~\cite{PikhurkoEtAl2019}, where some sufficient requirements are part of the general criteria stated there. We find it helpful though to state the exact requirements separately from the statements used to establish the stability result, as this is the place were the most can be achieved using ad-hoc arguments invoking the structure of $C$. This motivates the following definition.
\begin{definition} \label{def:reconstructur}
    A graph $T$ is an \emph{$\ell$-reconstructor} of a given graph $C$ if there exists a strong homomorphism from $G$ to $C$ for any graph $G$ satisfying the following:
    \begin{itemize} \setlength\itemsep{0em}
        \item[(i)] $T$ is an induced subgraph of $G$, that is there exists $S \subseteq V(G)$ such that $G[S] \cong T$.
        \item[(ii)] For any $S \subseteq V(G)$, $|S| \leq \ell$ there exists a strong homomorphism from $G[S]$ to $C$.
    \end{itemize}
\end{definition}

Let us introduce some additional notions. 
For a given graph $C$, we say that a graph $T$ \emph{uniquely embeds into $C$} if there exists a strong homomorphism $\psi: V(T) \to V(C)$ and $\psi$ is unique up to automorphism, that is for any additional strong homomorphism $\psi': V(T) \to V(C)$ there must exist $\varphi_T \in \textrm{Aut}(T)$ and $\varphi_C \in \textrm{Aut}(C)$ such that $\varphi_C \circ \psi' \circ \varphi_T \equiv \psi$.
For a set of vertices $X \subseteq V(C)$, we write $N_{C,X}(v) = N_C(v) \cap X$ for the neighborhood of  any $v \in V(C)$ in $X$. We let $\sim_X$ denote the equivalence relationship induced in $V(C)$ by $N_{C,X}$ and $[v]_X$ the equivalence class containing $v \in V(C)$. We say $X$ \emph{defines unique neighborhoods in $C$} if $[v]_X = \{v\}$ for all $v \in V(C)$.

It is easy to see that a graph $C$ is an $\ell$-reconstructor of itself if $|V(C)| \leq \ell-1$. More commonly though, the following lemma is used, which is also implicit in Theorem 4.1 in~\cite{PikhurkoEtAl2019}.
\begin{lemma} \label{lemma:reconstructor_old}
    For a given graph $C$ and set of vertices $X \subseteq V(C)$, the subgraph $C[X]$ is an $\ell$-reconstructor of $C$ if (i) $|X| \leq \ell - 2$, (ii) $C[X]$ uniquely embeds into $C$, and (iii) $X$ defines unique neighborhoods in $C$.
\end{lemma}

A more technical but stronger condition was previously already implicitly formulated in~\cite{PikhurkoVaughan_2013} to establish stability of the Clebsch graph for $g_{6,3}$ using a certificate with $m = 7$. Here we further strengthen it in the form of the following lemma.
\begin{lemma} \label{lemma:reconstructor}
    Let a graph $C$ be given. If there exists $X' \subseteq X \subseteq V(C)$ satisfying
    \begin{enumerate} \setlength\itemsep{0em}
        \item[(1)] either
        \begin{itemize} \setlength\itemsep{0em}
            \item[(a)] $|X| = |X'| \leq \ell - 1$ and $C[X]$ uniquely embeds into $C$ or
            \item[(b)] $|X| = \ell$, $|X'| \leq \ell-2$, and $C[X' \cup \{x\}]$ uniquely embeds into $C$ for any $x \in X$,
        \end{itemize}
        \item[(2)] $X'$ defines unique neighborhoods in $C$,
        \item[(3)] for any $v_1, v_2 \in V(C) \setminus X$ there exists $X'' \subseteq X$ with $|X''| \leq \ell-2$ such that
        \begin{itemize} \setlength\itemsep{0em}
            \item[(a)] if $v_1 = v_2$, then $C[X'' \cup \{v_1\}]$ uniquely embeds into $C$ and $[v_1]_{X''} = \{v_1\}$,
            \item[(b)] if $v_1 \neq v_2$, then $C[X'' \cup \{v_i\}]$ uniquely embeds into $C$ for some $i \in \{1, 2\}$, $[v_1]_{X''} \neq[v_2]_{X''}$ and the bipartite subgraph of $C$ induced by $[v_1]_{X''}$ and $[v_2]_{X''}$ is either complete or empty,
        \end{itemize}
    \end{enumerate}
    then $C[X]$ is an $\ell$-reconstructor of $C$.
\end{lemma}

\begin{proof}
   Let some arbitrary graph $G$ satisfying the assumptions of \cref{def:reconstructur} be given. Fix a copy of $C[X]$ in $G$ that is guaranteed to exist by assumption (i) of \cref{def:reconstructur}, that is fix $Y \subseteq V(G)$ such that $G[Y] \cong C[X]$ and let $\psi: Y \hookrightarrow X$ be the corresponding graph isomorphism. The goal is to construct a strong homomorphism $\varphi: V(G) \to V(C)$ satisfying $\varphi \vert_Y \equiv \psi$ given the requirements of the lemma and the assumptions of a reconstructor.
   
   Let us first assume that $|X| = |X'| \leq \ell - 1$, that is in particular $X = X'$. By assumption (ii) of \cref{def:reconstructur}, there exists a strong homomorphism $\varphi_v: Y \cup \{v\} \to V(C)$ for any vertex $v \in V(G)$ since $|Y \cup \{v\}| \leq \ell$. As $C[X]$ uniquely embeds into $C$ by (1a), there exist $\xi_v \in \textrm{Aut}(G[Y])$ and $\chi_v \in \textrm{Aut}(C)$ such that $\chi_v \circ \varphi_v \vert_{Y} \circ \xi_v \equiv \psi$. Let us therefore w.l.o.g. assume that $\varphi_v \vert_Y \equiv \psi$.
    Since $X$ defines unique neighborhoods in $C$, it follows that $\varphi_v (v) \in V(C)$ is uniquely defined for every $v \in V(G)$ and its adjacencies in $X$ match those of $v$ in $Y$, that is $\psi(N_{G,Y}(v)) = N_{C,X}(\varphi_v (v))$. Let $\varphi: V(G) \to V(C)$ therefore be given by $\varphi(v) = \varphi_v (v)$ for every $v \in V(G)$.
    
    Now if $|X| = \ell$ and $|X'| \leq \ell -2$, that is in particular $X' \subsetneq X$, then write $Y' = \psi^{-1}(X') \subset Y$. Let $x \in X \setminus X'$ be arbitrary but fixed and write $y = \psi^{-1}(x) \in Y \setminus Y'$. By assumption (ii) of \cref{def:reconstructur}, there exists a strong homomorphism $\varphi_{v}: Y' \cup \{y, v\} \to V(C)$ for any vertex $v \in V(G)$ since $|Y' \cup \{y, v\}| \leq \ell$ by (1b). As $C[X' \cup \{x\}]$ uniquely embeds into $C$ by (1b), we can again w.l.o.g. assume that  $\varphi_{v}\vert_{Y' \cup \{y\}} \equiv \psi\vert_{Y' \cup \{y\}}$.
    %there exist $\xi_v \in \textrm{Aut}(G[Y' \cup \{y\}])$ and $\chi_{v} \in \textrm{Aut}(C)$ such that $\chi_{v} \circ \varphi_{v}\vert_{Y' \cup \{y\}}  \circ \xi_v \equiv \psi\vert_{Y' \cup \{y\}}$.
    Since $X'$ defines unique neighborhoods in $C$ by (2), it likewise follows that $\varphi_{v} (v)$ is uniquely defined for every $v \in V(G)$ and its adjacencies in $X' \cup \{x\}$ match those of $v$ in $Y' \cup \{y\}$ under $\psi$. Since this is actually independent of our choice of $x$, we in fact also have that the adjacencies of $\varphi_{v} (v)$ in $X$ match those of $v$ in $Y$, that is $\psi(N_{G,Y}(v)) = N_{C,X}(\varphi_v (v))$ for any $v \in V(G)$. Let $\varphi: V(G) \to V(C)$ therefore again be given by $\varphi(v) = \varphi_{v} (v)$ for every $v \in G$.
    
    Let us now establish that in either case $\varphi$ is in fact a strong homomorphism from $G$ to $C$.
    For arbitrary $w_1, w_2 \in V(G)$, let $X'' \subset X$ be as given by (3) for $v_1 = \varphi(w_1) \in V(C)$ and $v_2 = \varphi(w_2) \in V(C)$ and write $Y'' = \varphi^{-1}(X'') = \psi^{-1}(X'') \subset Y$. The induced subgraph $G[Y'' \cup \{w_1, w_2\}]$ has at most $\ell$ vertices, so by assumption (ii) of \cref{def:reconstructur} there exists a strong homomorphism $\varphi': Y'' \cup \{w_1, w_2\} \to V(C)$. Since $G[Y'' \cup \{w_i\}] \cong C[X'' \cup \{v_i\}]$ for \emph{any} $i \in \{1,2\}$ by construction of $\varphi$ and since $C[X'' \cup \{v_i\}]$ uniquely embeds into $C$ for \emph{some} $i \in \{1,2\}$ by (3), we can w.l.o.g. assume that $\varphi' \vert_{Y''} \equiv \psi \vert_{Y''} \equiv \varphi \vert_{Y''}$. We distinguish the cases $v_1 = v_2$ and $v_1 \neq v_2$. In the former, $[v_1]_{X''} = \{v_1\}$ by (3a). Since $\varphi'$ is a strong homomorphism, we have $\varphi' (w_i) \in [v_1]_{X''}$ for $i \in \{1,2\}$ and it follows that $\varphi' (w_2) = v_2 = v_1 = \varphi' (w_1)$. Since $\varphi'$ is a strong homomorphism and $C$ simple, $w_1$ and $w_2$ are therefore not adjacent in $G$ if $v_1 = v_2$. Now if $v_1 \neq v_2$, then $[v_1]_{X''} \neq [v_2]_{X''}$ and the bipartite subgraph induced by $[v_1]_{X''}$ and $[v_2]_{X''}$ in $C$ is complete if $v_1$ and $v_2$ are adjacent in $C$ and empty if they are not by (3b). Since $\varphi'$ is a strong homomorphism, we have $\varphi' (w_1) \in [v_1]_{X''}$ and $\varphi' (w_2) \in [v_2]_{X''}$ and therefore $\varphi'(w_1)$ and $\varphi'(w_2)$ are adjacent in $C$ if and only if $v_1$ and $v_2$ are. Since $\varphi'$ is a strong homomorphism, $v_1$ and $v_2$ are therefore adjacent in $C$ if and only if $w_1$ and $w_2$ are  adjacent in $G$, establishing that $\varphi: V(G) \to V(C)$ is a strong homomorphism.
\end{proof}

Having established how to find reconstructors, let us now formally introduce the relevant notion of stability and how to establish it using a reconstructor. Let
\begin{equation*}
    \mB(C) = \{C[m_1, \ldots, m_n] : m_1, \ldots, m_n \in \NN_0 \} \subset \mG
\end{equation*}
denote the set of all blow-ups of $C$ and assume that we have established that $\lambda(\mG) = \lambda(\mB(C))$.
\begin{definition} \label{def:stability}
    We have \emph{perfect $C$-stability} for $\lambda$ on $\mG$ if there exists some $n_0$ such that for any graph $G \in \mathcal{G}_n$ with $n \geq n_0$, the number of edges that need to be changed in order to turn $G$ into an element of $\mB(C)$ is bounded by $n_0 \, (\lambda(G) - \lambda(\mG)) {n \choose 2}$.
\end{definition}
Perfect stability strengthens the standard notion of stability, where one requires that for every $\varepsilon > 0$ there exists $\delta > 0$ such that $\lambda(G) \leq \lambda(\mG) + \delta$ and $n \geq 1/\delta$ implies that one has to change at most $\varepsilon n^2$ edges to obtain a graph in $\mB(C)$. Clearly perfect stability also implies that any $G \in \mathcal{G}_n$ with $n \geq n_0$ satisfying $\lambda(G) = \lambda(n, \mG)$ is a blow-up of $C$.

\begin{theorem} \label{thm:stability}
    Assume we have the following:
    \begin{enumerate}  \setlength\itemsep{0em}
        \item A set of forbidden graphs $\mX$ in which each graph $X \in \mX$ defines unique neighborhoods in itself, $\mG = \mG(\mX)$, and $\lambda: \mG \to \RR$ satisfying \(\lambda(G) = \sum_{H \in \mG_n} d_H(G) \, \lambda(H)\) for any $G$ of large enough order $n$. 
        \item A graph $C \in \mG$ satisfying $\lambda(\mB(C)) = \lambda(\mG)$.
        \item A certificate $(m, (\tau_i), (\mQ^{(i)}))$ establishing $\lambda (\mG)$ in which all sharp graphs are in $\mB(C)$.
        \item An $m$-reconstructor $T$ of $C$ satisfying $\lambda(\mG(\mX \cup \{T\})) > \lambda(\mG)$.
    \end{enumerate}
    Then we have perfect $C$-stability for $\lambda$ on $\mG$.
\end{theorem}

\begin{proof}
    The proof is essentially identical to that of Theorem 4.1 and Theorem 5.13 in~\cite{PikhurkoEtAl2019}, where the properties of the reconstructor $T$ replace requirements (2a) and (2b) and the requirement that $\lambda(\mG(\mX \cup \{T\})) > \lambda(\mG(\mX))$ replaces (2c) in Theorem 4.1.
\end{proof}

\cref{thm:stability} in combination with \cref{lemma:reconstructor} allows us to establish our stability results. However, this only implies that extremal constructions must be \emph{some} blow-up of a given graph $C$ of order $n$. One can go further though and establish that \emph{a particular weighting} of the blow-up must be optimal. For both of our applications, that weighting will be the balanced one. For any vector of weights $\bw \in \DS_n$, where $\DS_n$ is the $n$-dimensional probability simplex, we define 
\begin{equation*}
    \mB_{\bw}(C) = \{C[m_1, \ldots, m_n] : | m_i - w_i \, ( m_1 + \ldots + m_n) | \leq 1 \text{ for all } 1 \leq i \leq n\} \subset \mB(C)
\end{equation*}
as the set of all blow-ups of $C$ where the parts are weighted according to $\bw$ where $w_i$ denotes the $i$-th entry of $\bw$.

While theoretically the problem of establishing $\argmin_{\bw \in \DS_n} \lambda(\mB_{\bw}(C))$ could be stated as a polynomial minimization problem over $n-1$ variables, the usual approach relies on studying if any of the matrices $\mQ^i$ have unique zero eigenvectors and then invoking the subsequent proposition to relate them to graph densities, assuming the associated types $\tau_i$ are present in $C$ and fulfill certain properties. 
We will need the following result,  c.f.~Lemma 3.4 in~\cite{PikhurkoEtAl2019}.  Before stating it, assume some $1 \leq i \leq r$ and $\psi: [v_i] \hookrightarrow V(C)$ us given, where $v_i$ is the order of $\tau_i$, such that $(C, \psi)$ is a $\tau_i$-flag and $w_j \neq 0$ for all $j \in \psi([v_i])$. Let $\delta^{\psi}_{F}(C, \bw)$ now denote the probability that $l_i - v_i$ vertices chosen not-necessarily-injectively and at random according to $\bw$ in $V(C) \setminus \psi([v_i])$ together with the vertices labelled by $\psi$ induce a flag isomorphic to $F$. Note that this generalizes the previous definition of $d^{\psi}_{F}(C)$, where the vertices were selected uniformly at random.

\begin{lemma}\label{lemma:zero_eigenvector}
    Assume we have the following:
    \begin{enumerate}  \setlength\itemsep{0em}
        \item A set of forbidden graphs $\mX$, $\mG = \mG(\mX)$, and $\lambda: \mG \to \RR$ satisfying \(\lambda(G) = \sum_{H \in \mG_n} d_H(G) \, \lambda(H)\) for any $G$ of large enough order $n$. 
        \item A certificate $(m, (\tau_i)_{i \in [r]}, (\mQ^{(i)})_{i \in [r]} )$ establishing $\lambda (\mG)$.
        \item A graph $C \in \mG_n$ and vertex weights $\bw \in \DS_n$ satisfying $\lambda (\mB_{\bw}(C)) = \lambda(\mG)$.
        \item Some $1 \leq i \leq r$ and $\psi: [v_i] \hookrightarrow V(C)$, where $v_i$ is the order of $\tau_i$, such that $(C, \psi)$ is a $\tau_i$-flag and $w_j \neq 0$ for all $j \in \psi([v_i])$.
    \end{enumerate}
    Then denoting all $\tau_i$-flags of order $l_i = (m + v_i)/2$ by $F_1, \ldots, F_{g_i}$ in the same order that they are used as indices for $\mQ^i$, the vector
    \begin{equation} \label{eq:limit_eigenvector}
        \bx = (\delta^{\psi}_{F_1}(C, \bw), \ldots, \delta^{\psi}_{F_{g_{i}}}(C, \bw))
    \end{equation}
    is a zero eigenvector of $\mQ^i$.
\end{lemma}

Using this, one can formulate some sufficient criteria to establish the uniqueness of a given optimal weighting. The following is a generalization of Lemma~6.2 in~\cite{PikhurkoEtAl2019} that is based on ideas previously used in~\cite{PikhurkoVaughan_2013}.

\begin{proposition} \label{thm:symmetry}
    Assume we have the following:
    \begin{enumerate}  \setlength\itemsep{0em}
        \item A set of forbidden graphs $\mX$, $\mG = \mG(\mX)$, and $\lambda: \mG \to \RR$ satisfying \(\lambda(G) = \sum_{H \in \mG_n} d_H(G) \, \lambda(H)\) for any $G$ of large enough order $n$. 
        \item A certificate $(m, (\tau_i)_{i \in [r]}, (\mQ^{(i)})_{i \in [r]} )$ establishing $\lambda (\mG)$.
        \item $C \in \mG_n$ and $\bw \in \DS_n$ satisfying $\lambda(\mB_{\bw}(C)) = \lambda(\mG)$ and $w_v \neq 0$ for all $v \in V(C)$.
        \item A sequence of sets $X_1, \ldots, X_k \subseteq V(C)$
        such that
        \begin{enumerate}
            \item $C[X_1]$ uniquely embeds into $C$ and $\lambda(\mG(\mX \cup \{C[X_1]\})) > \lambda(\mG)$,
            \item $X_i \subseteq X_1 \cup \bigcup_{j=1}^{i-1} \big\{v \in V(C) : [v]_{X_j} = \{v\} \big\} $ for any $1 \leq i \leq k$,
            \item $V(C) = \bigcup_{i=1}^{k} \big\{v \in V(C) : [v]_{X_i} = \{v\} \big\}$,
            \item the type $\tau_{i}$ is a labelled version of $C[X_i]$ for any $1 \leq i \leq k$,
            \item the matrix $\mQ^{i}$ is of co-rank $1$ for any $1 \leq i \leq k$.
        \end{enumerate}
    \end{enumerate}
    Then $\bw$ is the unique minimizer of $\lambda(\mB_{\bw'}(C))$ for $\bw' \in \DS_n$.
\end{proposition}

\begin{proof}
    Let $\bw' \in \DS_n$ satisfy $\lambda(\mB_{\bw'}(C)) = \lambda(\mG)$. Since $\lambda(\mG(\mX \cup \{C[X_1]\})) > \lambda(\mG)$ and $C[X_1]$ uniquely embeds into $C$, it follows that w.l.o.g. $w_v' \neq 0$ for any $v \in X_1$. Let us now inductively argue over $1 \leq i \leq k$ that in fact $w'_v = w_v$ for any $v \in V(C)$ satisfying $[v]_{X_i} = \{v\}$. We can apply \cref{lemma:zero_eigenvector} since $X_i \subseteq X_1 \cup \bigcup_{j=1}^{i-1} \big\{v \in V(C) : [v]_{X_j} = \{v\} \big\} $ and therefore $w_v' \neq 0$ for any $v \in X_i$ by inductive assumption. Let $\bx_i$ and $\bx_i'$ therefore be as given for $\bw$ and $\bw'$, where $\psi: [v_i] \hookrightarrow X_i$ is chosen such that $(C[X_i], \psi) = \tau_i$ and where $v_i$ is the order of $\tau_i$. Since $Q^{i}$ has co-rank $1$ by assumption, it follows that in fact $\bx_i = \bx_i'$. Consider the $\tau_i$-flag $F_v$ consisting $l_i - v_i$ isolated vertices, where $l_i = (m + v_i)/2$, connected to $\tau_i$ according to the neighbors of $v$ to $X_i$ in $C$. Using $F_v$ to index $\bx_i$, we note that $[\bx_i]_{F_v} = w_v^{l_i-v_i}$ since $v$ has a unique neighborhood in $X_i$. It follows that $w_v = w'_v$ for any $v \in \big\{v \in V(C) : [v]_{X_i} = \{v\} \big\}$ and since $V(C) = \bigcup_{i=1}^k \{v \in V(C): [v]_{X_i} = \{v\} \big\}$, it follows inductively that $\bw' = \bw$.
\end{proof}

\subsection{Practical aspects of using the flag algebra approach} \label{sec:flag_algebra_practical}

{\tt flagmatic}, developed by Emil Vaughan and hosted on \href{https://www.github.com/jsliacan/flagmatic}{\nolinkurl{github.com/jsliacan/flagmatic}}, calculates graph densities and passes SDP formulations to solvers like {\tt CSDP}\cite{csdp} and {\tt SDPA}\cite{sdpa}. In order to obtain rigorous mathematical proofs, the floating point-based solutions from the SDPs need to be rounded to (fractional) values while ensuring the matrices remain positive semi-definite. {\tt flagmatic} handles this rounding process and produces verifiable certificates for the proofs, allowing anyone with access to the software to verify the results independently. The certificates for our lower bounds can be found at \href{https://doi.org/10.5281/zenodo.6364588}{\nolinkurl{doi.org/10.5281/zenodo.6364588}}. We do not supply certificates of stability and symmetry from~\cite{PikhurkoEtAl2019}, as they are incompatible with our ad-hoc arguments in \cref{lemma:reconstructor} and \cref{thm:stability}. Below we describe some specifics regarding the bounds for each problem.

\paragraph{Bounds and stability for $c_{s,t}$.} For the lower bounds of $c_{3,4}$ and $c_{3,5}$, we used $m = 6$ and were able to reduce the number of types to $4$ for each problem. The certificates for these proofs are contained in the files {\tt c\_34.json} and {\tt c\_35.json}.

In order to derive stability for $c_{3,4}$, we also obtained a certificate for $m = 7$ and verified that all sharp graphs are those with a strong homomorphism into the Schl\"afli graph. The certificate for this proof is contained in the file {\tt c\_34\_7.json}. Perfect stability now follows from \cref{thm:stability} using the uniquely embeddable subgraph $T$ of $C_{S}$ consisting of two triangles joined by an edge. The fact that $T$ is a $7$-reconstructor of the Schl\"afli graph follows by computationally verifying the requirements of \cref{lemma:reconstructor}. The certificate for the proof that $\lambda(\mG(\{T\}) > \lambda(\mG)$ is contained in the file {\tt c\_34\_reconstructor.json}. The fact that the balanced blow-up of the Schl\"afli graph is the unique optimal weighting follows by applying \cref{thm:symmetry}. Here let the Schl\"afli graph be as ordered in the {\tt graph6} representation stated in \cref{sec:constructions} with vertices labelled $1, \ldots, 27$ and use $X_1 = \{1, 2, 4, 6, 7\}$, $X_2 = \{1, 2, 12, 13, 14\}$ and $X_3 = \{1, 2, 4, 6, 12\}$. The certificate for the proof that $\lambda(\mG(C_S[X_1])) > \lambda(\mG)$ is contained in the file {\tt c\_34\_symmetry.json}.

Regarding the non-differentiability of $c_{3,4}(x)$ at $x = 41\cdot 3^{-6}$, we considered a differently weighted version of \cref{eq:cst}. For $w_s, w_t \geq 0$ satisfying $w_s + w_t = 2$, let
\begin{equation} \label{eq:cstw}
    c^{(w_s,w_t)}_{s,t} = \lim_{n \to \infty} \min \left\{ w_s \, \frac{k_s(\overline{G})}{{n \choose s}} +  w_t \, \frac{k_t(G)}{{n \choose t}} : |G| = n \right\}
\end{equation}
and note that clearly $c_{s,t} = c^{(1,1)}_{s,t}$. We used flagmatic to establish a lower bound for $c^{(1-\epsilon,1+\epsilon)}_{3,4}$ that matches the upper bound given by the Schl\"afli graph when $\epsilon = 10^{-4}$, implying the non-differentiability of $c_{3,4}(x)$ at that point. The certificate for this proof is contained in the file {\tt c\_34\_epsilon.json} and uses $m=6$ as well as the same types as the proof of $c_{3,4}$.
    
For the lower bounds of $c_5$ and $c_{4,5}$, we used $m = 8$ and the values obtained after rounding are reasonably close to those indicated by the SDP solver. The certificates for these proofs are contained in the files {\tt c\_5.json} and {\tt c\_45.json}.

\paragraph{Bounds and stability for $g_{s,t}$.} For $g_{4,5}$, we used $m = 7$ and also used the full $38$ types to establish this bound as the rounding failed when attempting to reduce the number of types involved. In general, this problem seemed more demanding and for example required the double precision solver. The certificate for the proof is contained in the file {\tt g\_45.json}. We can also easily derive perfect stability from this certificate by verifying that all sharp graphs are those with a strong homomorphism into $C_{R(3,5)}$ and noting that the graph $T$ of order $5$ obtained by joining two triangles at a vertex is a $7$-reconstructor of $C_{R(3,5)}$ by \cref{lemma:reconstructor_old}. The certificate for the proof that $\lambda(\mG(\{K_5, K_4^+\}) > \lambda(\mG(\{K_5\})$ is contained in the file {\tt g\_45\_reconstructor.json}. The fact that the balanced blow-up of $C_{R(3,5)}$ is the unique optimal weighting also easily follows by applying \cref{thm:symmetry} with $T$ as $C[X_1]$ and $k=1$ after verifying that the associated matrix $\mQ$ indeed has co-rank $1$.

All of the reported lower bounds for $g_{s,t}$ when $(s,t) \neq (4,5)$ were obtained using $m = 8$. The certificates for these proofs are contained in the files {\tt g\_st.json}.

\section{Remarks and open problems} \label{sec:discussion}

\subsection{Symmetric Ramsey multiplicity}

The crucial question regarding the Ramsey multiplicity of $K_4$ is whether it whether there exists a graph whose blow-up sequence determines its value. Our answer to the $c_{3,4}$ and $c_{3,5}$ problem supports this possibility for $c_4$. Our efforts for \cref{thm:4-ramsey_multiplicity} however show that this construction could be unexpectedly complex.

In any case, the deeper understanding of the generating sets of our Cayley graphs on 192, 384, and 768 vertices, to the extent that they can be generalized, for arbitrary large $k$, to (semi-)direct products of the $3$-element cyclic group $\ZZ_3$ with $k$ copies of $\ZZ_2$, would be extremely interesting. Besides that they would of course improve the upper bound on $c_4$, more importantly they might shed light on how to count cliques in blow-ups of certain Cayley graphs. This could also lead to identifying the right choice of generating sets and help in constructing new Cayley graph \emph{sequences} that improve the value of $c_t$ for $t\geq 5$. 
This would be particularly desirable, since the best known general upper bound for the Ramsey multiplicity problem still comes from the original construction of Thomason~\cite{Thomason_1989}, the analysis of which was improved  by Jagger, {\v{S}}{\'t}ov{\'\i}{\v{c}}ek, and Thomason~\cite{jagger1996multiplicities} to $c_t < 0.835 \cdot 2^{1-{t\choose 2}}$ for $t \geq 7$ via a direct counting of the homogeneous sets in the blow-up with Fourier analytic methods. 

For large $t$ R\"odl conjectured (cf.~\cite{Franek_2002}) that $c_t 2^{{t\choose 2}} \rightarrow 0$, possibly even exponentially fast. Yet, no $t$ is known where the value of the quotient is less than $1/2$. The best known general lower bound is due to Conlon~\cite{Conlon_2012} who showed that $c_t \ge C^{-t^2 (1+o(1))}$ where $C \approx 2.18$. Conlon comments, that his proof is the Ramsey multiplicity analogue of the Erd\H os-Szekeres proof for the symmetric Ramsey number. Hence it seems plausible that the approach of the recent improvements of Campos, Griffiths, Morris, and Sahasrabudhe~\cite{campos2023exponential} on the upper bound of the symmetric Ramsey number could be adapted to improve Conlon's constant $C$, though likely not to the best possible $\sqrt{2}$.

\subsection{Asymmetric Ramsey multiplicities}
Additionally, we explored graphs with small clique densities of order $s$ and given independent set densities of order $t$. It is a tantalizing open problem to determine, or just obtain a better idea about the lower bounding curve $c_{3,4}(x)$ of the region of realizable $K_3$- and $\bar{K}_4$-density pairs in graphs. Given the results in~\cref{fig:c_34}, even formulating a conjecture seems challenging. The properties of the curve, such as the number of non-differentiable points, what constructions they would be determined by, and convexity or concavity, remain unclear. The argument of Bollob\'as~\cite{bollobas_2004} giving a piece-wise linear lower bound for $c_{2,t}(x)$ does not seem to generalize easily.

Nikiforov~\cite{Nikiforov_2001} established that $g_{s,t} = (t-1)^{1-s}$ holds only for a finite number of pairs $s,t \geq 3$, that is for all but a finite number of such pairs the construction given by Tur\'an graphs is not optimal. Das et al.~\cite{DasEtAl_2013} established that equality does not hold for $s > 3$ and arbitrary $t$ as well as for $t \geq 2074$ when $s = 3$. The results of Pikhurko and Vaughan~\cite{PikhurkoVaughan_2013} establish that equality holds when $s = 3$ and $t \leq 7$. The following proposition shows that there must exist a smallest $7 < t_0 \leq 2074$ such that $g_{3,t} = 1/(t-1)^2$ for $t < t_0$ and $g_{3,t} < 1/(t-1)^2$ for $t \geq t_0$. It would be interesting to precisely determine this value and to therefore determine all pairs of $s,t$ for which Erd\H{o}s' original intuition about the Tur\'an graph holds true.
\begin{proposition} \label{prop:g_st_add_vertex}
    Given $s, t_0 \geq 2$, we have $g_{s,t} \leq ( t - t_0 + g_{s,t_0}^{1/(1-s)} )^{1-s}$ for all $t \geq t_0$. In particular, if $g_{s,t_0} < (t_0-1)^{1-s}$ then $g_{s,t} < (t-1)^{1-s}$ for all $t \geq t_0$
\end{proposition}
\begin{proof}
    Let $\gamma = g_{s,t}^{1/(s-1)}$ and $(G_n)_{n \in \mathbb{N}}$ a sequence of graphs on $n$ vertices with clique number at most $t_0-1$ satisfying $\lim_{n \to \infty} k_s(\overline{G_n}) / \binom{n}{s} = g_{s,t_0}$. By adding an independent set of $\lfloor \gamma n \rfloor$ vertices to each $G_n$ and fully connecting it to all other vertices, we get a sequence of graphs with clique number at most $t_0$ that establishes
    \begin{equation*}
        g_{s,t_0+1} \leq (g_{s,t_0} + \gamma^s)/(1+\gamma)^{s} = (g_{s,t_0}^{1/(1-s)} + 1)^{1-s}.
    \end{equation*}
    Recursively applying this gives us the desired result.
\end{proof}

\subsection{Learning-based optimization} \label{sec:ml-approach}

Recent interest has emerged in using Machine Learning approaches to tackle combinatorial optimization problems. Specifically, Wagner~\cite{Wagner_2021} proposed finding constructions in Extremal Combinatorics by framing the underlying optimization problems within a Reinforcement Learning context, an idea similar to the {\lq}active learning{\rq} approach suggested earlier by Bello et al.~\cite{BelloEtAl_2016}. In the context of round-based games and Reinforcement Learning, states are constructed bit-by-bit, using a Neural Network to decide for each bit whether it should be $1$ or $0$ based on previous decisions. After sampling several discrete states through this procedure, the Neural Network's parameters are updated to produce states closer to the best ones sampled. The update procedure in~\cite{Wagner_2021} employs the Cross-Entropy (CE) method by Rubinstein~\cite{Rubinstein_1999, RubinsteinKroese_2004}.

We experimented a fair amount with Reinforcement Learning-based methods, including different network architectures and learning methods from the ones suggested in~\cite{Wagner_2021}, but found their performance inferior to both SA and Tabu Search (TS) in terms of solution quality and computation time for medium and large-sized problems:
\begin{itemize}\setlength{\itemsep}{0pt}
    \item Finding the $13$-vertex graph for $g_{4,5}$ takes TS 1s, SA 10s, and CE around 8min.
    \item Finding the (presumably unique) Cayley graph in $\ZZ_3 \times \ZZ_2^{\times 6}$ giving an upper bound of $c_4 < 0.03027$ takes TS around 1 min, SA around 4 min, and CE just under one hour.
    \item For the $27$-vertex graph for $c_{3,4}$ TS tends to get stuck in local minima, only very rarely finding it in around 30s, SA consistently finds it in 30s, and CE did not find it in any reasonable timeframe.
    \item For $c_4$ CE also found meaningful improvement over the previous best upper bound by constructing Cayley graphs on $192$ vertices, but the corresponding value of $0.03022$ is weaker than what TS found in the same search space and significantly weaker than what TS found on $768$ vertices, a search space where CE could not feasibly be run.
\end{itemize}

 While the values above do no constitute fully rigorous and definitive benchmarking of the three methods, they do reflect a fair amount of (hyper)parameter tuning as well as our anecdotal experience for these particular problems.  
 We find it notable that for the problems studied here Tabu Search, with its straightforward idea and implementation, emerged as perhaps the most preferable method, fairly reliably finding good solutions in a  short amount of time while requiring essentially zero tuning. 
 
  The idea of pretraining-free learning-based meta-heuristics using neural networks has a long history~\cite{HopfieldTank_1985, Smith_1999}, but seems to so far not yet outperform simple baselines. There is a fair amount of ongoing research focused on improving optimization methods by training their behavior on dedicated training datasets, which are hoped to reflect real-world applications prior to the actual usage of these methods. However, this approach strongly differs from our use case, where we are predominantly interested in solving singular, difficult problems rather than a large pool of small or moderately-sized ones.  
 It should be further noted that an RL-based approach is inherently wasteful in our particular context, that is for attacking simple black-box discrete optimization problems given by \cref{eq:opt_problem}; it models a distribution over the states $\{0,1\}^N$ through $N$ consecutive evaluations of a neural network. This round-based approach is artificially imposed, since, unlike with playing chess or Go, we do not have any opponent's choices to respond to and our distribution will ultimately collapse to a single state. This effectively means that rather than trying to find (an approximation of) a strategy, i.e., a still very large sub-tree of the full game tree, we are merely looking for a single path from the root to a leaf.

\subsection{Classifying common and uncommon graphs}

Before it was disproved, the conjecture of Erd\H{o}s mentioned in the introduction was generalised to arbitrary graphs by Burr and Rosta~\cite{BurrRosta_1980}.
For graphs $H$ and $G$ we let $t(H,G)$ be the number of copies of $H$ in $G$ and
\[ c(H) = \lim_{n \to \infty} \min\{ t(H,G)+t(H,\overline{G}) : |G|=n \} / \binom{n}{|H|} \, . \]
We note that $c_t = c(K_t)$ and while Burr and Rosta conjectured that the minimum is always attained by the random graph, we say that a graph $H$ is \emph{common} if $c(H) = 2^{1-|E(H)|}$ and \emph{uncommon} otherwise.
A triangle is common by Goodman's result and Thomason's~\cite{Thomason_1989} result implies that $K_t$ is uncommon for $t \ge 4$.
Since then large families of common and uncommon graphs have been found and for an overview we refer to the recent paper of Grzesik, Lee, Lidicky, and Volec~\cite{GrzesikEtAl_2020} and the references therein.
Only recently the first uncommon graphs $H$ were found, for which $c(H)$ is known~\cite{fox2023ramsey}.
Also an off-diagonal variant of common and uncommon pairs of graphs was studied by Behague, Morrison, and Noel~\cite{behague2022common, behague2023off}.
Lower bounds for common graphs are usually obtained through the flag algebra approach and upper bounds rely on constructions that beat $2^{1-|E(H)|}$.
There is also an interesting connection to the famous and still open conjecture of Sidorenko~\cite{Sidorenko_1993}, which implies that every bipartite graph is common.

\subsection{Parallels to problems in Additive Combinatorics}

Similar problems to the ones  considered in this paper have been studied in Additive Combinatorics, where one is interested in minimizing the number of monochromatic solutions to some system of linear equations in a coloring of the integers $[n] = \{1, \ldots, n\}$, the cyclic group $\mathbb{Z}_n$ or the repeated product of a fixed small finite field $\mathbb{F}_q^n$. Graham, R\"{o}dl and Ruci\'nski~\cite{GrahamRodlRucinski_1996} asked this question regarding Schur triples in two-colorings of $[n]$, which was independently resolved by Datskovsky~\cite{Datskovsky_2003}, Robertson and Zeilberger~\cite{RobertsonZeilberger_1998}, and Schoen~\cite{Schoen_1999}, who showed that the true answer lies far below the number expected in a random coloring. In contrast to this, Cameron, Cilleruelo and Serra~\cite{CameronCillerueloSerra_2007} showed that in finite groups the number of monochromatic solutions to any equation in an odd number of variables, which includes Schur triples, is minimized by the random coloring. Wolf~\cite{Wolf_2010} as well as Lu and Peng~\cite{LuPeng_2012} studied the question of how many $4$-term arithmetic progressions a two-coloring of $\mathbb{Z}_n$ can contain. Interestingly, the upper bounds given for this type of problem likewise consist of a type of {\lq}blow-up{\rq} of a finite construction, indicating that the methods used here are also applicable in this context. In fact, denoting the minimum fraction of monochromatic $k$-term arithmetic progressions in a $2$-coloring of $\ZZ_n$ by $m_k(\ZZ_n)$, it is easy to derive the following lemma from their work (see the proof of Lemma~4.1 and Theorem~1.5 in~\cite{LuPeng_2012}).
\begin{lemma}
\label{lemma:additive_blow-up}
    Let $k \geq 3$ be an integer and $A$ be a partial coloring of $\ZZ_m$, where $\ell$ distinct elements in arithmetic progression are uncolored for some $\ell$ dividing $m$. If for all colorings of these $\ell$ elements there are no monochromatic $k$-APs in $\ZZ_m$ using both elements colored in $A$ and in the coloring of the $\ell$ elements, then for any integer $t$ we have
    \[ m_k(tm) \le m_k(A) - m_k(\ell) \left(\tfrac{\ell}{m}\right)^2 + m_k(t \ell) \left(\tfrac{\ell}{m}\right)^2 \, , \]
    where $m_k(A)$ denotes the maximum fraction of monochromatic $k$-term arithmetic progressions for any coloring of the $\ell$ remaining elements in $A$.
    In particular, if $m_k(A)=1/n$ and therefore $m_k(\ell) = 1/\ell$, we get
    \[ \liminf_{n\to\infty} m_k(\ZZ_n) \le \frac{1}{n+\ell} \, . \]
\end{lemma}
Lu and Peng~\cite{LuPeng_2012} found a construction for $k=4$ with $n=11$ and $\ell=1$ giving $\liminf m_4(\ZZ_n) \le 1/12$. Using Tabu Search we found the following partial colorings of $\ZZ_{44}$
\begin{align*}
    \star1101111011\star1000101110 \\
    \star0010000100\star0111010001
\end{align*}
and of $\ZZ_{226}$
\begin{align*}
    \star01111001000001011110111001111101101110100011101001010011\\
    00110101101000111010001001000001100010000101111101100001%
    \\
    \star10000110111110100001000110000010010001011100010110101100\\
    11001010010111000101110110111110011101111010000010011110
\end{align*}
that together with \cref{lemma:additive_blow-up} establish $\liminf_{n \to \infty} m_5(\ZZ_n) \leq 1/48$ as well as $\liminf_{n \to \infty} m_6(\ZZ_n) \leq 1/228$ with the former improving upon a bound found by Lu and Peng~\cite{LuPeng_2012}, who also conjectured their construction for $k=4$ to be tight. We believe that our constructions for $k=5, 6$ provide tight bounds as well.

Lastly, we note that Saad and Wolf~\cite{SaadWolf_2017} also initiated a more systematic study of the question when the answer is given by random colorings,  with recent results by Fox, Pham and Zhao~\cite{FoxPhamZhao_2021}, Kam\v{c}ev, Liebenau and Morrison~\cite{KamcevLiebenauMorrison_2021b, KamcevLiebenauMorrison_2021a} as well as Versteegen~\cite{versteegen2021common, versteegen2023linear}. Ru{\'e} and the third author also recently extended the flag algebra framework to additive problems in vector spaces over finite fields \cite{rue2023rado}.

\bigskip

\noindent \textbf{Note.} Recently McKay~\cite{mckay-personal} experimented with various local search strategies in order to further improve the upper bound on $c_4$.
Starting from our best Cayley graph on $768$ vertices, his search iteratively switches the edge/non-edge status of certain pairs of vertices. The most successful attempt created a graph with value $10486266368/768^4=0.0301422734319$, which is an improvement over Theorem~\ref{thm:4-ramsey_multiplicity} of the order $10^{-6}$. 
McKay's graph has $768$ vertices, 148724 edges, 536 vertices of degree 387, 232 vertices of degree 388, and, unlike our construction, has a trivial automorphism group.

\bigskip

\noindent \textbf{Acknowledgements.}  This work was partially funded by the Deutsche Forschungsgemeinschaft (DFG, German Research Foundation) under Germany’s Excellence Strategy – The Berlin Mathematics Research Center MATH+ (EXC-2046/1, project ID: 390685689).

\end{document}